\documentclass[a4paper]{amsart}
\usepackage[utf8]{inputenc}
\usepackage[T1]{fontenc}
\usepackage{amssymb,amsthm,amsmath}
\usepackage[british]{babel}
\usepackage{calrsfs}
\usepackage{hyperref}
\usepackage{mathbbol}
\usepackage{mathabx}
\usepackage{fancyvrb}
\usepackage{seqsplit}

\theoremstyle{plain}
\newtheorem{theorem}[subsection]{Theorem}
\newtheorem{proposition}[subsection]{Proposition}
\newtheorem{corollary}[subsection]{Corollary}	
\theoremstyle{remark}
\newtheorem{remark}[subsection]{Remark}

\newcommand{\noproof}{\hfill \qed}
\newcommand{\comp}{\raisebox{0.2mm}{\ensuremath{\scriptstyle{\circ}}}}
\newcommand{\defn}{\textbf}
\newcommand{\tensor}{\ensuremath{\otimes}}
\newcommand{\links}{\lgroup}
\newcommand{\rechts}{\rgroup}

\newcommand{\V}{\ensuremath{\mathcal{V}}}

\newcommand{\K}{\ensuremath{\mathbb{K}}}
\newcommand{\N}{\ensuremath{\mathbb{N}}}
\newcommand{\Q}{\ensuremath{\mathbb{Q}}}
\newcommand{\Z}{\ensuremath{\mathbb{Z}}}

\newcommand{\Liek}{\ensuremath{\mathsf{Lie}}_{\K}}
\newcommand{\Alg}{\ensuremath{\mathsf{Alg}}}
\newcommand{\qLie}{\ensuremath{\mathsf{qLie}}}
\newcommand{\qLiek}{\ensuremath{\mathsf{qLie}}_{\K}}
\newcommand{\nAlg}{\ensuremath{\text{$n$-$\mathsf{Alg}$}}}
\newcommand{\Vect}{\ensuremath{\mathsf{Vect}}}
\newcommand{\Fun}{\ensuremath{\mathsf{Fun}}}
\newcommand{\Mag}{\ensuremath{\mathsf{Mag}}}
\newcommand{\Set}{\ensuremath{\mathsf{Set}}}
\newcommand{\Top}{\ensuremath{\mathsf{Top}}}

\newcommand{\kar}{\ensuremath{\mathrm{char}}}

\newcommand{\LACC}{{\rm (LACC)}}

\begin{document}

\title[A characterisation of Lie algebras]{A characterisation of Lie algebras\\ via algebraic exponentiation}

\author{Xabier García-Martínez}
\address[Xabier García-Martínez]{Departamento de Matemáticas, Esc.\ Sup.\ de Enx.\ Informática, Campus de Ourense, Universidade de Vigo, E--32004, Ourense, Spain\newline
and\newline
Faculty of Engineering, Vrije Universiteit Brussel, Pleinlaan 2, B--1050 Brussel, Belgium}
\email{xabier.garcia.martinez@uvigo.gal}

\author{Tim Van~der Linden}
\address[Tim Van~der Linden]{Institut de
Recherche en Math\'ematique et Physique, Universit\'e catholique
de Louvain, che\-min du cyclotron~2 bte~L7.01.02, B--1348
Louvain-la-Neuve, Belgium}
\thanks{}
\email{tim.vanderlinden@uclouvain.be}

\thanks{This work was partially supported by Ministerio de Economía y Competitividad (Spain), grant number MTM2016-79661-P. The first author is a Postdoctoral Fellow of the Research Foundation--Flanders (FWO). The second author is a Research
Associate of the Fonds de la Recherche Scientifique--FNRS}

\begin{abstract}
In this article we describe varieties of Lie algebras via \emph{algebraic exponentiation}, a concept introduced by Gray in his Ph.D.\ thesis. For $\K$ an infinite field of characteristic different from $2$, we prove that the variety of Lie algebras over~$\K$ is the only variety of non-associative $\K$-algebras which is a non-abelian \emph{locally algebraically cartesian closed \LACC} category. More generally, a variety of $n$-algebras~$\V$ is a non-abelian \LACC\ category if and only if $n=2$ and $\V=\Liek$. In characteristic $2$ the situation is similar, but here we have to treat the identities $xx=0$ and $xy=-yx$ separately, since each of them gives rise to a variety of non-associative $\K$-algebras which is a non-abelian \LACC\ category.
\end{abstract}

\subjclass[2010]{08C05, 17A99, 18B99, 18A22, 18D15}
\keywords{Lie algebra; locally algebraically cartesian closed category; algebraic exponentiation}

\maketitle

\section{Introduction}
\emph{Which kind of a topology should a function space carry?} This is a historically important problem in topology, whose solution clearly depends on the properties the function space should have. A standard requirement is a correspondence between the continuous maps into a function space and those out of a cartesian product.

The classical answer of Fox~\cite{Fox} is \emph{the compact-open topology}: if $B$, $X$ and $Y$ are topological spaces, where $B$ is regular and locally compact, and $Y^B$ carries the compact-open topology, then there exists a natural bijection between the continuous maps $h\colon{B\times X\to Y}$ and those of the form $h^*\colon{X\to Y^B}$. Category theory expresses this by saying that the functor $B\times (-)\colon{\Top\to \Top}$ is left adjoint to the functor $(-)^B\colon{\Top\to \Top}$; the space $B$ is then called \emph{exponentiable}. 

In that same article it is also explained that while the conditions on $B$ can be weakened, it is impossible to entirely remove them. That is to say, non-exponentiable spaces exist, and the category $\Top$ of all topological spaces and continuous maps between them is not \emph{cartesian closed}: certain functors $B\times (-)$ do not admit a right adjoint. Later the work of Day and Kelly~\cite{Day-Kelly} gave a precise characterisation of those spaces which are exponentiable. This led to further developments where the problem was considered in slightly different contexts, as well as in situations which are quite far from topology. 

Indeed, cartesian closedness is a property with important consequences, which may be considered whenever there are cartesian products. As explained in~\cite{Eilenberg-Kelly, MacLane} its definition easily extends to more general settings involving any kind of a monoidal product. Many examples have been considered in the literature; for instance, the categories of sets, categories, or cocommutative coalgebras over a commutative ring are cartesian closed, while categories of vector spaces are closed with respect to the tensor product. 

Recently, in the work of Gray~\cite{GrayPhD, Gray2012} and Bourn--Gray~\cite{Bourn-Gray}, a different kind of closedness was considered, which is meant to be more appropriate to non-abelian algebraic contexts~\cite{Jo1990,ClHoJa}. Here the cartesian product functor $B\times (-)$ is replaced by a functor~$B\flat (-)$ which happens not to be induced by a monoidal product. However, both the examples and the general theory seem to indicate that in the setting of \emph{semi-abelian categories}~\cite{JaMaTh}, this is the right thing to consider~\cite{BJK,Bourn-Janelidze:Semidirect, acc, GaVa-LACC, BeBou, CiMaMe, dvdl1}. 

The aim of our present article is not to give an overview of this theory---for that we refer the reader to~\cite{GaVa-LACC} and the references there---but rather to prove one specific result: amongst all varieties of non-associative algebras over an infinite field, there exists precisely one variety that forms a non-abelian category whose objects are exponentiable with respect to the functors $B\flat (-)$, namely the category of Lie algebras. (As we shall see, there is actually one exception to this rule: when the characteristic of the field is $2$, there are two solutions, since we then have to distinguish between the identities $xx=0$ and $xy=-yx$.) Here, by a \defn{non-associative algebra} we mean a vector space equipped with a bilinear operation. 

In a variety of algebras $\V$ which is a semi-abelian category---for instance, $\V$ could be any variety of non-associative algebras---for any object $B$ the functor $B\flat (-)\colon{\V\to \V}$ takes an object $X$ and sends it to the kernel $B\flat X$ of the universally induced morphism $\links 1_B\;0\rechts\colon B+X\to B$. As explained in~\cite{BJK,Bourn-Janelidze:Semidirect}, functors of this kind play a key role in the description of \emph{internal actions} (which correspond to split extensions via a semi-direct product construction) as algebras for a monad. It also appears in a characterisation of the concept of a \defn{locally algebraically cartesian closed (LACC)} category~\cite{GrayPhD, Gray2012, Bourn-Gray}. The variety $\V$ is such, if and only if each functor $B\flat (-)\colon{\V\to \V}$, equivalently,
\begin{enumerate}
\item is a left adjoint;
\item preserves all colimits; or
\item preserves binary sums.
\end{enumerate}
This condition has some strong categorical-algebraic consequences, as explored in the above-mentioned papers and in~\cite{acc}.

In the article~\cite{GrayLie}, Gray proved that the variety $\Liek$ of Lie algebras over a field~$\K$ is a~\LACC\ category. (Actually, he proved it for Lie algebras over a commutative ring with unit.) In our article~\cite{GaVa-LACC} we managed to find a partial converse to this theorem: we showed that the condition \LACC\ characterises Lie algebras amongst $\K$-algebras with an alternating multiplication ($xx=0$). The result is a simple categorical description of the Jacobi identity. We did, however, fail to remove the requirement that the identity $xx=0$ must hold. At the same we failed to understand why a subvariety of the variety of Lie algebras cannot be a locally algebraically cartesian closed category, unless its algebras carry the zero product. (In other words: unless $xy=0$ is an identity, which is equivalent to saying that the category is abelian.) 

The purpose of our present paper is to clarify those two issues. We present a new and self-contained analysis of the condition \LACC\ for categories of non-associative algebras, with the concrete aim of proving that over an infinite field $\K$ of characteristic different from $2$, there is precisely one non-abelian example: the category $\Liek$ of Lie algebras over $\K$.

In Section~\ref{Section Preliminaries} we start by recalling some definitions and basic properties of varieties of non-associative algebras, with in particular an analysis of the objects $B\flat X$ in this context. This leads to Theorem~\ref{Theorem previous} which provides a summary of the main result of~\cite{GaVa-LACC}. Section~\ref{Section Commutative} is devoted to the commutative case. We actually prove that \emph{there is no commutative case}: if $\K$ is an infinite field of characteristic different from $2$ and~$\V$ is a variety of commutative $\K$-algebras that forms a \LACC\ category, then $\V$ is an abelian category (Theorem~\ref{Theorem Commutative}). As a consequence we obtain Corollary~\ref{Corollary degree 2}, which says that any variety of $\K$-algebras satisfying an identity of degree~$2$ is a subvariety of~$\Liek$ as soon as it is a \LACC\ category. In characteristic~$2$ we prove a version of Corollary~\ref{Corollary degree 2} for the variety $\qLiek$ of \defn{quasi-Lie algebras}, which by definition satisfies the Jacobi identity and the identity $xy=yx=-yx$ (Corollary~\ref{Corollary qLie}).

Section~\ref{Section General Case} contains this paper's first major result, Theorem~\ref{Theorem degree 3}. It says that, in characteristic zero, Corollary~\ref{Corollary degree 2} does in fact describe the general case, since \emph{a variety of non-associative algebras which doesn't satisfy any non-trivial identity of degree $2$ cannot be a \LACC\ category}. The proof of this result involves a system of~128 polynomial equations $(f_i=0)_{1\leq i \leq 128}$---see Appendix~\ref{Appendix}. In order to show that the system is inconsistent, we had to rely on a computer algebra system; we chose to use an open-source package called \textsc{Singular}~\cite{DGPS}. In Section~\ref{Section Prime Characteristic} we extend this to infinite fields of prime characteristic, adapting the approach in characteristic zero to a proof of Theorem~\ref{Theorem degree 3 prime}. 

In Section~\ref{Section Subvarieties} we combine these results with a proof that \emph{a variety of non-associative $\K$-algebras which is strictly smaller than $\Liek$ can only form a \LACC\ category when it forms an abelian category} in order to obtain Theorem~\ref{Theorem Lie}: \emph{a variety of non-associative $\K$-algebras which forms a non-abelian \LACC\ category is either the variety of Lie algebras over~$\K$ (when $\kar(\K)\neq 2$) or one of the varieties $\Liek$ and $\qLiek$ (when $\kar(\K)= 2$)}. This result is then further sharpened in Section~\ref{Section n-Ary} where we prove that a variety of $n$-algebras viewed as a non-abelian semi-abelian category can only be \LACC\ when $n=2$---so that it is~$\Liek$ or $\qLiek$. This is our main result, Theorem~\ref{Theorem n-Ary}. 

To conclude the article, in Section~\ref{Section Final Remarks} we make a number of final remarks on potential ways of extending our results to more general situations.

\section{Preliminaries}\label{Section Preliminaries}

\subsection{Non-associative algebras}
Let $\K$ be a field. A \defn{(non-associative) algebra $A$ over $\K$} is a $\K$-vector space equipped with a bilinear operation ${A\times A\to A}$. We write $\Alg_{\K}$ for the category of non-associative algebras over~$\K$, with linear maps between them that preserve the operation. A \defn{variety} of non-associative algebras is any equationally defined class of algebras, considered as a full subcategory~$\V$ of~$\Alg_{\K}$. 

The variety $\Liek$ of \defn{Lie algebras} over $\K$ is the variety of non-associative algebras with a multiplication that is \defn{alternating} ($xx = 0$) and such that the \defn{Jacobi identity}
\[
{x(yz)+z(xy)+ y(zx)=0}
\] 
holds. Next to the Jacobi identity, instead of being alternating, the objects in the variety $\qLiek$ of \defn{quasi-Lie algebras}~\cite{quasiLie} over $\K$ satisfy \defn{anticommutativity} ($xy =- yx$).
As long as the characteristic of the field $\K$ is different from $2$, these two varieties coincide. However, when $\kar(\K)=2$, the variety $\Liek$ is strictly smaller than $\qLiek$.

If $\V$ satisfies commutativity $xy=yx$ instead of being alternating, it is the variety of so-called \defn{mock--Lie algebras}---see~\cite{MR3598575} for a survey on this type of algebras.

As is common for Lie algebras, we call an algebra \defn{abelian} when it has a zero product: $xy = 0$. These algebras form a variety which is isomorphic to the category of vector spaces over $\K$.

\subsection{The homogeneous components of a non-associative polynomial}
In order to state Theorem~\ref{Theorem Homogeneity} (which is Corollary~2 on page~8 of~\cite{Shestakov}) we need to fix some terminology. For a given set $S$, a \defn{(non-associative) polynomial} with variables in $S$ is an element of the free $\K$-algebra on $S$. Recall that the left adjoint ${\Set\to \Alg_{\K}}$ factors as a composite of the \emph{free magma} functor $M\colon{\Set\to \Mag}$ with the \emph{magma algebra} functor $\K[-]\colon {\Mag\to \Alg_{\K}}$. The elements of $M(S)$ are non-associative words in the alphabet $S$, and the elements of $\K[M(S)]$, the polynomials, are $\K$-linear combinations of such words. A \defn{monomial} in $\K[M(S)]$ is any scalar multiple of an element of~$M(S)$. The \defn{type} of a monomial $\varphi(x_{1},\dots,x_{n})$ is the element $(k_{1},\dots,k_{n})\in \N^{n}$ where $k_{i}$ is the number of times that $x_{i}$ appears in~$\varphi(x_{1},\dots,x_{n})$, and the \defn{degree} of a monomial is the sum of all $k_i$.

A polynomial is \defn{homogeneous} if its monomials are all of the same type; then its \defn{degree} is the degree of its monomials. Any polynomial may thus be written as a sum of homogeneous polynomials, which are called its \defn{homogeneous components}. 
A homogeneous polynomial $\phi(x_1,\dots,x_n)$ is called \defn{multilinear} when its monomials $\psi(x_{i_1},\dots,x_{i_m})$ have type $(1,\dots,1)\in \N^m$. A \defn{multilinear identity} is an identity of the form $\phi(x_{1},\dots,x_{n})=0$, where $\phi(x_{1},\dots,x_{n})$ is a multilinear polynomial.
In varieties of Lie algebras we know that:

\begin{theorem}[Theorem~3 in Section~4.2 of~\cite{Bahturin}]\label{Theorem Linearity General}
In a subvariety of $\Liek$, any nontrivial identity has a nontrivial multilinear consequence.\noproof
\end{theorem}

From now on, throughout the article we shall assume that $\K$ is an infinite field, so that we can use the next result.

\begin{theorem}[\cite{Shestakov, Bahturin}]\label{Theorem Homogeneity}
If $\V$ is a variety of algebras over an infinite field, then all of its identities are of the form $\phi(x_{1},\dots,x_{n})=0$, where $\phi(x_{1},\dots,x_{n})$ is a polynomial, each of whose homogeneous components $\psi(x_{i_{1}},\dots, x_{i_{m}})$ again gives rise to an identity $\psi(x_{i_{1}},\dots, x_{i_{m}})=0$ in $\V$. \noproof
\end{theorem}

\subsection{The objects \texorpdfstring{$B\flat X$}{BbX}} 
Let $B$ and $X$ be free $\K$-algebras. The object $B\flat X$, being the kernel of the morphism $\links 1_{B}\;0\rechts\colon {B+X\to B}$, consists of those polynomials with variables in $B$ and in $X$ which can be written in a form where all of their monomials contain variables in $X$. For instance, given $b$, $b'\in B$ and $x\in X$, the expression $(b(xx))b'$ is allowed, but~$bb'$ is not.

If now $\V$ is a variety of $\K$-algebras, and $B$ and $X$ are free algebras in $\V$, then the situation changes only slightly, as follows. We may view the elements of $B\flat X$ as polynomials with variables in $B$ and in $X$ satisfying the condition mentioned above, modulo those identities which hold in $\V$ that are expressible in terms of such polynomials. Abusing terminology, we make no distinction between the polynomials and the equivalence class to which they belong. See \cite{GaVa-LACC} for further details on this issue. 

\subsection{The condition \texorpdfstring{\LACC}{(LACC)}}
As explained in~\cite{Gray2012}, the condition \LACC\ may be expressed by asking that the canonical comparison map
\[
\links B\flat \iota_{X}\; {B\flat \iota_{Y}\rechts \colon B\flat X + B\flat Y \to B\flat (X+Y)}
\]
is an isomorphism. We shall only ever need this when the algebras under consideration are free. In particular, from now on $B$ will denote the free $\V$-algebra on two generators $b$, $b'$, and $X$ and~$Y$ will denote the free $\V$-algebras on one generator $x$ and $y$, respectively. We write $b_1$ and $b_1'$, $b_2$ and $b_2'$, when we consider a second copy of $B$, so that $\links B\flat \iota_{X}\; {B\flat \iota_{Y}\rechts }$ sends $b_1$ and $b_2$ to $b$ and $b_1'$ and $b_2'$ to $b'$. 

Thus, the structure of the algebra $B\flat (X+Y)$ is relatively simple, since it is the subalgebra of the sum $B + X + Y$---the free $\V$-algebra on four generators $b$, $b'$, $x$,~$y$---generated by the monomials containing an~$x$ or a $y$. 

On the other hand, the algebra $B\flat X + B\flat Y$ is harder to understand. It is the $\V$-algebra generated by the words formed by elements that belong to $B\flat X$ and~$B\flat Y$, formally, an $x$ surrounded by a finite number of $b_1$ and $b'_1$ or a $y$ surrounded by a finite number of $b_2$ and $b'_2$. Then, for example, the word $(b_1x)y$ belongs to $B\flat X + B\flat Y$, but $b_i(xy)$ does not (unless $xy = 0$, so that each algebra in $\V$ has a zero product, which makes $\V$ an abelian category). 

Moreover, in par with the principle mentioned above, when dealing with identities in $B\flat X + B\flat Y$, we need to be careful about what we use as input in the polynomials that express the identities of $\V$. For instance, if $\V$ is the variety of associative algebras, then it is not true that $(xb_1)y = x(b_2y)$ in $B\flat X + B\flat Y$. However, the identity $\big(y(xb_1)\big)y = y\big((xb_1)y\big)$ does hold. Another example: if $(xb_1)y$ is zero in $B \flat X + B\flat Y$, then either the identity $pq = 0$ holds in $\V$, or $0\in \{xb_1,y\}$; however, both cases imply that $\V$ is an abelian category.

\subsection{\texorpdfstring{\LACC}{(LACC)} and the Jacobi identity}
From the paper \cite{GaVa-LACC} we need the following result, which we shall prove here for the sake of completeness.

\begin{theorem}\label{Theorem GaVa-LACC}\label{Theorem previous}
Let $\K$ be an infinite field and $\V$ a variety of non-associative algebras over $\K$. If $\V$ is a \LACC\ category, then there exist $\lambda_{1}$, \dots, $\lambda_{8}$, $\mu_{1}$, \dots, $\mu_{8}$ in $\K$ such that
\begin{align*}
z(xy)=
\lambda_{1}(zx)y&+\lambda_{2}(xz)y+
\lambda_{3}y(zx)+\lambda_{4}y(xz)\\
&+\lambda_{5}(zy)x+\lambda_{6}(yz)x+
\lambda_{7}x(zy)+\lambda_{8}x(yz)
\end{align*}
and 
\begin{align*}
(xy)z=
\mu_{1}(zx)y&+\mu_{2}(xz)y+
\mu_{3}y(zx)+\mu_{4}y(xz)\\
&+\mu_{5}(zy)x+\mu_{6}(yz)x+
\mu_{7}x(zy)+\mu_{8}x(yz)
\end{align*}
are identities in $\V$. If, furthermore, $xy+yx=0$ holds in $\V$, then the Jacobi identity is an identity in $\V$.
\end{theorem}
\begin{proof}
Since the comparison map $\links B\flat \iota_{X}\; {B\flat \iota_{Y}\rechts \colon B\flat X + B\flat Y \to B\flat (X+Y)}$ is an isomorphism, the element $b(xy)$ in $B \flat (X + Y)$ must come from an element of ${B \flat X + B \flat Y}$. Therefore, by Theorem~\ref{Theorem Homogeneity} there must exist a homogeneous identity involving $b(xy)$ that does not involve any of the monomials $b(yx), (xy)b$ or $(yx)b$. This is the first identity in the statement of the theorem.

If now $xy=-yx$ in $\V$, then we can reduce the given identities to 
\[
z(xy) = \lambda y(zx) + \mu x(yz),
\]
for some $\lambda$, $\mu \in \K$. Considering first $y = z$, and then $x = z$, we deduce that either $\lambda = \mu = -1$, or $z(zx) = 0$ is an identity of $\V$. The first case is exactly the Jacobi identity. In the second case, we see that 
\[
0 = (x + y)((x+y)b) = x(yb) + y(xb).
\]
Therefore, the comparison map sends $x(yb_2) - y(xb_1) \in B\flat X + B\flat Y$ to zero in $B \flat (X + Y)$, which via \LACC\ and Theorem~\ref{Theorem Homogeneity} implies that each algebra in $\V$ has a zero product.
\end{proof}

\section{The commutative case}\label{Section Commutative}
The aim of this section is to exclude, as much as possible, the case of commutative algebras. Here the characteristic of the field $\K$ plays a role: when $\kar(\K)=2$, the equations $xx=0$ and $xy+yx=0$ are no longer equivalent, so need to be treated separately. On the other hand, then commutative = anticommutative. 

We prove that if $\V$ is any variety of $\K$-algebras satisfying an identity of degree~2, then it is a subvariety of $\Liek$ as soon as it is a \LACC\ category. In particular then, its algebras are anticommutative. 

\begin{theorem}\label{Theorem Jacobi Jordan}
If $\K$ is an infinite field of characteristic different from $2$, then a variety of mock--Lie algebras over $\K$ is a \LACC\ category if and only if it is an abelian category. 
\end{theorem}
\begin{proof}
Let $\V$ be a variety of mock--Lie algebras over $\K$, which by definition are commutative and satisfy the Jacobi identity. Suppose that $\V$ is a \LACC\ category. On one hand we know that 
\begin{align*}
(bb')(xy) &= - \big((bb')x)\big)y - \big((bb')y)\big)x \\
{} &= \big((bx)b'\big)y + \big((b'x)b\big)y + \big((by)b'\big)x + \big((b'y)b\big)x
\end{align*}
is an identity of $\V$; on the other hand, 
\begin{align*}
(bb')(xy) &= - \big(b(xy)\big)b' - \big(b'(xy)\big)b \\
{} &= \big((bx)y\big)b' + \big((by)x\big)b' + \big((b'x)y\big)b + \big((b'y)x\big)b \\
{} &= -(bx)(yb') - \big((bx)b'\big)y - (by)(b'x) - \big((by)b'\big)x \\
&\qquad\qquad- (b'x)(by) - \big((b'x)b\big)y - (b'y)(bx) - \big((b'y)b\big)x.
\end{align*}
Therefore, using that $1+1\neq 0$ in $\K$, we find that
\begin{align*}
\big((bx)b'\big)y + \big((b'x)b\big)y + \big((by)b'\big)x + \big((b'y)b\big)x + (bx)(b'y) + (b'x)(by) = 0
\end{align*}
is an identity in $\V$. Hence since $\V$ is \LACC, in the sum $B\flat X + B\flat Y$ we have
\begin{multline*}
\big((b_1x)b_1'\big)y + \big((b_1'x)b_1\big)y + \big((b_2y)b_2'\big)x + \big((b_2'y)b_2\big)x \\+ (b_1x)(b_2'y) + (b_1'x)(b_2y) = 0.
\end{multline*}
This equality holds in the free algebra $B+X+B+Y$, so that by Theorem~\ref{Theorem Homogeneity}, $(b_1x)(b_2'y)=0$ is an identity in $\V$. Now the element $(b_1x)(b_2'y)$ of $B\flat X + B\flat Y$ is sent by $\links B\flat \iota_{X}\; {B\flat \iota_{Y}\rechts \colon B\flat X + B\flat Y \to B\flat (X+Y)}$ to $(bx)(b'y)$, which is thus zero in~$B\flat (X + Y)$. As a consequence, $(b_1x)(b_2'y)$ is zero in $B\flat X + B\flat Y$. Hence, either $pq=0$ is an identity in $\V$, or $0\in \{b_1x,b_2'y\}$ in one of the free algebras ${B+X}$ or~${B+Y}$. Therefore each algebra in $\V$ has a zero product, which makes $\V$ an abelian category.
\end{proof}

\begin{theorem}\label{Theorem Commutative}
Let $\K$ be an infinite field of characteristic different from $2$. Let $\V$ be a variety of commutative $\K$-algebras. If $\V$ is a \LACC\ category, then it is an abelian category.
\end{theorem}
\begin{proof}
By Theorem~\ref{Theorem previous} we know that $\V$ has to satisfy an identity of the form 
\[
z(xy) = \lambda (zx)y + \lambda' (zy)x.
\]
Therefore
\begin{align*}
\lambda (bx)y + \lambda' (by)x = b(xy) = b(yx) = \lambda (by)x + \lambda' (bx)y,
\end{align*}
so in $B\flat X + B\flat Y$ we have
\begin{align*}
\lambda (b_1x)y + \lambda' (b_2y)x = \lambda (b_2y)x + \lambda' (b_1x)y.
\end{align*} 
Hence $\lambda (b_1x)y - \lambda' (b_1x)y=0$ by Theorem~\ref{Theorem Homogeneity}, 
and $\lambda = \lambda'$.
By the computation
\begin{align*}
(bx)y = \lambda b(xy) + \lambda (by)x = \lambda^2(bx)y + \lambda^2(by)x + \lambda (by)x,
\end{align*}
we know that 
\[
(\lambda^2 - 1)(b_1x)y + (\lambda^2 + \lambda) (b_2y)x = 0
\]
holds in $B\flat X + B\flat Y$, so $\lambda$ must be $-1$. Then $\V$ is a variety of mock--Lie algebras. Since it is a \LACC\ category, by Theorem~\ref{Theorem Jacobi Jordan} it is an abelian category.
\end{proof}

\begin{corollary}\label{Corollary degree 2}
If $\K$ is an infinite field of characteristic different from $2$. Let $\V$ be a variety of $\K$-algebras satisfying a non-trivial identity of degree~$2$. If $\V$ is a \LACC\ category, then it is a subvariety of~$\Liek$.
\end{corollary}

\begin{proof}
There are essentially only two types of non-trivial degree $2$ identities. In characteristic different from $2$, the equation $xx = 0$ is equivalent to $xy + yx = 0$, so without loss of generality we can assume that the variety $\V$ under consideration satisfies an identity of the form $xy + \lambda yx = 0$, with $\lambda \in \K$.
In this case,
\begin{align*}
0 = xy + \lambda yx = xy + \lambda (- \lambda xy) = xy - \lambda^2 xy = (1 - \lambda^2) xy,
\end{align*}
so unless $\lambda$ equals to $1$ or $-1$, each algebra in $\V$ has a zero product, which makes $\V$ an abelian category. By the previous theorem, the same happens when $\lambda = -1$. The remaining case ($\lambda = 1$) is covered by Theorem~\ref{Theorem previous}, which establishes that $\V$ must be a subvariety of $\Liek$.
\end{proof}

Infinite fields of characteristic $2$ need special care, because of the difference between the identities $xx=0$ and $xy+yx=0$: the first implies the second, but not the other way round. Recall that $\Liek$ denotes the variety of $\K$-algebras satisfying the Jacobi identity and $xx=0$, and $\qLiek$ the variety of $\K$-algebras satisfying the Jacobi identity and $xy=-yx$ (which here is equivalent to $xy=yx$). Note that $\Liek$ is a subvariety of $\qLiek$ which is always strictly smaller.

\begin{theorem}\label{Theorem JJAlg char 2}
For any field $\K$ of characteristic $2$, the variety $\qLiek$ is a \LACC\ category.
\end{theorem}
\begin{proof}
The proof of Theorem~2.2 in~\cite{GrayLie} can be repeated word for word, with the exception of the passage which shows that $(ff)_n(w)=0$. Rather, a similar argument may be given that proves $(fg)_n(w)+(gf)_n(w)=0$.
\end{proof}

\begin{corollary}\label{Corollary qLie}
Let $\K$ be an infinite field of characteristic $2$. Let $\V$ be a variety of $\K$-algebras satisfying a non-trivial identity of degree~$2$. If $\V$ is a \LACC\ category, then it is a subvariety of $\qLiek$.
\end{corollary}
\begin{proof}
This is a straightforward variation on the proof of Corollary~\ref{Corollary degree 2}.
\end{proof}

\section{The general case}\label{Section General Case}

We now extend the above to the case when the variety does not satisfy any degree~2 identities (besides, of course, the multiples of the trivial identity $xy = xy$). We first treat the case of fields of characteristic zero.

\begin{theorem}\label{Theorem degree 3}
Let $\K$ be a field of characteristic zero. If a variety of non-associative $\K$-algebras does not satisfy any non-trivial identities of degree~$2$, then it cannot form a~\LACC\ category.
\end{theorem}
\begin{proof}
Assume that $\V$ is a variety of non-associative algebras with all of its non-trivial identities of degree~$3$ or higher. If $\V$ is \LACC, then by Theorem~\ref{Theorem previous} we know that 
there exist $\lambda_{1}$,~\dots,~$\lambda_{8}$, $\mu_{1}$,~\dots,~$\mu_{8}$ in $\K$ such that
\begin{align*}
z(xy)=
\lambda_{1}(zx)y&+\lambda_{2}(xz)y+
\lambda_{3}y(zx)+\lambda_{4}y(xz)\\
&+\lambda_{5}(zy)x+\lambda_{6}(yz)x+
\lambda_{7}x(zy)+\lambda_{8}x(yz)
\end{align*}
and 
\begin{align*}
(xy)z=
\mu_{1}(zx)y&+\mu_{2}(xz)y+
\mu_{3}y(zx)+\mu_{4}y(xz)\\
&+\mu_{5}(zy)x+\mu_{6}(yz)x+
\mu_{7}x(zy)+\mu_{8}x(yz)
\end{align*}
are identities in $\V$.
As we shall see, the isomorphism
\[
\links B\flat \iota_{X}\; B\flat \iota_{Y}\rechts \colon{B\flat X + B \flat Y\to B\flat (X + Y)}
\]
will force the coefficients $\lambda_i$ and $\mu_i$ to satisfy a certain system of polynomial equations. The strategy of our proof is to use the hypothesis that there are no non-trivial identities of degree~$2$ in $\V$ in order to find enough equations for the set of solutions of this system to be empty. This then means that \LACC\ cannot hold in $\V$. 

First we compute
\begin{align*}
(bx)y = \mu_1 (yb)x &+ \mu_2 x(yb) + \mu_3 x(yb) + \mu_4 x(by)\\ &+ \mu_5 (yx)b + \mu_6 (xy)b + \mu_7 b(yx) + \mu_8 b(xy).
\end{align*}
Further decomposing the last four monomials, we obtain an identity
\begin{align*}
0 = f_1 (bx)y &+ f_2 (xb)y + f_{3}y(bx) + f_{4}y(xb)\\ &+ f_{5}(by)x + f_{6}(yb)x + f_{7}x(by) + f_{8}x(yb)
\end{align*}
in $\V$, where the coefficients $f_i$ are taken in the ring of (commutative) polynomials
\[
R=\K[\lambda_1, \dots, \lambda_8, \mu_1, \dots, \mu_8].
\]
An explicit description of these coefficients is given in Appendix~\ref{Appendix}.

In $B\flat X + B \flat Y$, the polynomial
\begin{align*}
f_1 (b_1x)y &+ f_2 (xb_1)y + f_{3}y(b_1x) + f_{4}y(xb_1) \\ &+ f_{5}(b_2y)x + f_{6}(yb_2)x + f_{7}x(b_2y) + f_{8}x(yb_2)
\end{align*}
is zero, because the comparison map $\links B\flat \iota_{X}\; B\flat \iota_{Y}\rechts \colon{B\flat X + B \flat Y\to B\flat (X + Y)}$ is an isomorphism. Since $\V$ does not satisfy any equation of degree~$2$, and since in ${B\flat X + B \flat Y}$ all the monomials of the given polynomial have a degree of at most two---for instance, $(b_1x)y$ is a word formed by two letters: $b_1x$ and $y$---these monomials are linearly independent. Therefore, all of the $f_i$ must be zero.

We may proceed in the same way, decomposing the monomials $(xb)y$, $y(bx)$ and~$y(xb)$, thus obtaining additional polynomials $f_9$, \dots, $f_{32}$ in $R$ that need to vanish. (Again, see Appendix~\ref{Appendix}.) These thirty-two equations $(f_i=0)_{1\leq i\leq 32}$ might already seem a rather large system of polynomial equations, but unfortunately its set of solutions is not empty, so it does not give us the desired contradiction. Hence, to complete our proof, we need to enlarge the system by finding additional equations.

Let us now consider the monomial $(bb')(xy)$. On one hand, it is equal to 
\begin{align*}
\lambda_1 \big((bb')x\big)y &+ \lambda_2 \big(x(bb')\big)y + \lambda_3 y\big((bb')x\big) + \lambda_4 y\big(x(bb')\big) \\
&+ \lambda_5 \big((bb')y\big)x + \lambda_6 \big(y(bb')\big)x + \lambda_7 x\big((bb')y\big) + \lambda_8 x\big(y(bb')\big).
\end{align*}
We can decompose each of the inner brackets to see that
\begin{align*}
(bb')(xy) = \varphi_1(b, b', x)y + y\varphi_2(b, b', x) + \varphi_3(b, b', y)x + x\varphi_4(b, b', y),
\end{align*}
where
\begin{align*}
\varphi_i(b, b', z) = g_{i,1} (bz)b' &+ g_{i,2} (zb)b' + g_{i,3}b'(bz) + g_{i,4}b'(zb)\\ &+ g_{i,5}(b'z)b + g_{i,6}(zb')b + g_{i,7}b(b'z) + g_{i,8}b(zb')
\end{align*}
for some $g_{i,j}\in R$.

On the other hand, $(bb')(xy)$ is equal to
\begin{align*}
\mu_1 \big((xy)b\big)b' &+ \mu_2 \big(b(xy)\big)b' + \mu_3 b'\big((xy)b\big) + \mu_4 b'\big(b(xy)\big) \\
&+ \mu_5 \big((xy)b'\big)b + \mu_6 \big(b'(xy)\big)b + \mu_7 b\big((xy)b'\big) + \mu_8 b\big(b'(xy)\big).
\end{align*}
Expanding first the inner bracket, and then the outer one, we obtain
\begin{multline*}
\psi_1(b, b', x)y + y\psi_2(b, b', x) + \psi_3(b, b', y)x + x\psi_4(b, b', y) \\
+ \psi_5((b,x), (b',y))+ \psi_6((b,y) , (b',x))
\end{multline*}
where
\begin{align*}
\psi_i(b, b', z) = h_{i,1} (bz)b' &+ h_{i,2} (zb)b' + h_{i,3}b'(bz) + h_{i,4}b'(zb)\\ &+ h_{i,5}(b'z)b + h_{i,6}(zb')b + h_{i,7}b(b'z) + h_{i,8}b(zb')
\end{align*}
for some $h_{i,j}\in R$, while 
\begin{align*}
\psi_5((b,x), (b',y)) = f_{33} (b'y)(bx) &+ f_{34} (yb')(bx) + f_{35} (bx)(b'y) \\
&+ f_{36} (bx)(yb') + f_{37} (b'y)(xb)\\
& + f_{38} (yb')(xb) + f_{39} (xb)(b'y)\\
& + f_{40} (xb)(yb')
\end{align*}
and
\begin{align*}
\psi_6((b, y), (b', x)) = f_{41} (b'x)(by) &+ f_{42} (xb')(by)+ f_{43} (by)(b'x)\\
& + f_{44} (by)(xb') + f_{45} (b'x)(yb) \\
&+ f_{46} (xb')(yb)+ f_{47} (yb)(b'x)\\
& + f_{48} (yb)(xb'),
\end{align*}
for some $f_{33}$, \dots, $f_{48}\in R$.

Thus we have obtained the identity 
\begin{align*}
0 = \big(\psi_1(b, b', x)&-\varphi_1(b, b', x)\big)y + y\big(\psi_2(b, b', x) - \varphi_1(b, b', x)\big) \\
{} &+ \big(\psi_3(b, b', y)-\varphi_1(b, b', y)\big)x + x\big(\psi_4(b, b', y)-\varphi_1(b, b', y)\big) \\
{} &+ \psi_5((b,x), (b',y)) + \psi_6((b, y), (b', x))
\end{align*}
in $B \flat (X + Y)$, which gives rise to the identity
\begin{align*}
0 = \big(\psi_1(b_1, b'_1, x)&-\varphi_1(b_1, b'_1, x)\big)y + y\big(\psi_2(b_1, b'_1, x) - \varphi_1(b_1, b'_1, x)\big) \\
{} &+ \big(\psi_3(b_2, b'_2, y)-\varphi_1(b_2, b'_2, y)\big)x + x\big(\psi_4(b_2, b'_2, y)-\varphi_1(b_2, b'_2, y)\big) \\
{} &+ \psi_5((b_1,x), (b'_2,y)) + \psi_6((b_2, y), (b'_1, x))
\end{align*}
in $B \flat X + B\flat Y$. Theorem~\ref{Theorem Homogeneity} decomposes it into four separate identities. 

First of all, the polynomials $\psi_5((b_1,x), (b'_2,y))$ and $\psi_6((b_2, y), (b'_1, x))$ must vanish. Since there are no non-trivial identities of degree~$2$ in $\V$, their monomials are linearly independent, so $f_{33}$, \dots, $f_{48}$ are all zero.

Next, $\big(\psi_1(b_1, b'_1, x)-\varphi_1(b_1, b'_1, x)\big)y + y\big(\psi_2(b_1, b'_1, x) - \varphi_1(b_1, b'_1, x)\big)$ has to be zero in $B\flat X + B\flat Y$. Again, since there are no non-trivial identities of degree~2 in~$\V$, both $\big(\psi_1(b_1, b'_1, x)-\varphi_1(b_1, b'_1, x)\big)y$ and $y\big(\psi_2(b_1, b'_1, x) - \varphi_1(b_1, b'_1, x)\big)$ have to vanish. The first one may be written as
\begin{multline*}
0 = \big(f_{49} (xb_1)b_1' + f_{50} (b_1x)b_1'+ f_{51} b_1'(b_1x) + f_{52}b_1'(b_1x) \\
{} + f_{53} (xb_1')b_1+ f_{54} (b_1'x)b_1 + f_{55} b_1(xb_1')+ f_{56} b_1(b_1'x)\big)y,
\end{multline*}
where $f_{49}$, \dots, $f_{56}$ belong to $R$. In order for this to hold, we need that
\begin{multline*}
0= f_{49} (xb)b' + f_{50} (bx)b'+ f_{51} b'(bx) + f_{52}b'(bx) \\
+ f_{53} (xb')b+ f_{54} (b'x)b + f_{55} b(xb')+ f_{56} b(b'x)
\end{multline*}
is an identity of $\V$. So in $B\flat (X + Y)$ we have 
\begin{multline*}
0 = f_{49} (bx)y + f_{50} (xb)y+ f_{51} y(xb) + f_{52}y(xb) \\
{} + f_{53} (by)x+ f_{54} (yb)x + f_{55} x(by)+ f_{56} x(yb).
\end{multline*}
Using \LACC\ we see that 
\begin{multline*}
0 = f_{49} (b_1x)y + f_{50} (xb_1)y+ f_{51} y(xb_1) + f_{52}y(xb_1) \\
{} + f_{53} (b_2y)x+ f_{54} (yb_2)x + f_{55} x(b_2y)+ f_{56} x(yb_2)
\end{multline*}
in $B\flat X + B\flat Y$. Since the monomials in this expression are linearly independent, all of the polynomials $f_{49}$, \dots, $f_{56}$ in $R$ need to vanish.

Essentially the same type of reasoning applies to $y\big(\psi_2(b, b', x) - \varphi_1(b, b', x)\big)$, $\big(\psi_3(b, b', y)-\varphi_1(b, b', y)\big)x$ and $x\big(\psi_4(b, b', y)-\varphi_1(b, b', y)\big)$, yielding twenty-four additional polynomials $f_{57}$, \dots, $f_{80} \in R$ that have to be zero---see Appendix~\ref{Appendix}. This ends our analysis of the identities coming from decomposing $(bb')(xy)$, which allowed us to find forty-eight polynomials in $R$ that need to vanish. A similar study can be done, decomposing $(xy)(bb')$ in the same way, and we can obtain forty-eight additional polynomials $f_{81}$, \dots $f_{128} \in R$.

Let $I$ be the ideal of $R=\K[\lambda_1, \dots, \lambda_8, \mu_1, \dots, \mu_8]$ generated by $f_{1}$, \dots, $f_{128}$. A~common root for all these polynomials will be a root for any polynomial of~$I$. When we compute a Gr\"obner basis of $I$ with the computer algebra system \textsc{Singular}~\cite{DGPS}, the system tells us that $1 \in I$. More precisely, it says that we can obtain~$1$ as a linear combination in $\Q[\lambda_1, \dots, \lambda_8, \mu_1, \dots, \mu_8]$ of the polynomials $f_{1}$, \dots, $f_{128}$. As a consequence, $I = \K[\lambda_1, \dots, \lambda_8, \mu_1, \dots, \mu_8]$ for any field $\K$ of characteristic zero, hence the set of common roots of $f_{1}$, \dots, $f_{128}$ is empty. This completes the proof.
\end{proof}

\begin{remark}\label{code}
Here is the code we used to make the computer algebra package \textsc{Singular}~\cite{DGPS} show that the system of equations $(f_i=0)_{1\leq i\leq 128}$ has no solutions.
\begin{Verbatim}[numbers=left]
ring r=0,(x(1..8),y(1..8)),dp;
poly f(1) = y(5)*y(5) + y(6)*y(1) + y(7)*x(5) + y(8)*x(1) - 1;
poly f(2) = y(5)*y(6) + y(6)*y(2) + y(7)*x(6) + y(8)*x(2);
...
poly f(128) = - y(7)*y(8) - y(8)*x(8) + x(5)*y(5)*y(3) 
+ x(6)*x(5)*y(3) + x(7)*y(5)*x(3) + x(8)*x(5)*x(3) + x(5)*y(7)*y(7) 
+ x(6)*x(7)*y(7) + x(7)*y(7)*x(7) + x(8)*x(7)*x(7);
ideal i=f(1..128);
ideal si=std(i);
si;
\end{Verbatim}
The parameter \verb+dp+ in line 1 means that degree reverse lexicographical ordering of polynomials is chosen. The arXiv version of this paper includes an ancillary file containing the whole code we used, as well as an explicit way to write $1$ as a linear combination of the $(f_i)_i$ in $\Q[\lambda_1, \dots, \lambda_8, \mu_1, \dots, \mu_8]$.
\end{remark}

\begin{remark}
We want to highlight the importance of choosing the ``right'' monomial order when performing this calculation. While calculating the Gröbner basis and explicit linear combination for the ideal generated by the polynomials of Appendix~\ref{Appendix} with the degree reverse lexicographical order takes around 45 minutes on a recent laptop, the supercomputer \emph{Finisterrae~II} located in the facilities of the \emph{Centro de Supercomputación de Galicia} (CESGA)~\cite{Finisterrae} took 89 hours to perform the same task with respect to the degree lexicographical order. When the lexicographical order was chosen, it could not even calculate the Gröbner basis after 96 hours of computation.
\end{remark}

\begin{corollary}\label{Corollary subvariety}
Let $\K$ be a field of characteristic zero. A variety of non-associative $\K$-algebras is a variety of Lie algebras as soon as it forms a \LACC\ category.
\end{corollary}

\begin{proof}
Corollary~\ref{Corollary degree 2} and Theorem~\ref{Theorem degree 3} cover all the possibilities.
\end{proof}

\section{Fields of prime characteristic}\label{Section Prime Characteristic}

We now extend the reasoning of Theorem~\ref{Theorem degree 3} to algebras over a field of prime characteristic. First of all, note that for Theorem~\ref{Theorem Homogeneity} to work, we need to restrict ourselves to infinite fields. Next, we need to be careful with the rational coefficients that appear in the Gr\"obner basis of the ideal $I$ in the proof of Theorem~\ref{Theorem degree 3}. Indeed, the characteristic of the field could divide one of the denominators in those coefficients. 

\begin{theorem}\label{Theorem degree 3 prime}
Let $\K$ be an infinite field of prime characteristic. If a variety of non-associative $\K$-algebras does not satisfy any non-trivial identities of degree~$2$, then it cannot be a \LACC\ category.
\end{theorem}
\begin{proof}
In the proof of Theorem~\ref{Theorem degree 3}, instead of writing $1$ as a linear combination of the $(f_i)_{1\leq i\leq 128}$ in $\Q[\lambda_1, \dots, \lambda_8, \mu_1, \dots, \mu_8]$, we may find some $m\in \N$ as a linear combination of the $(f_i)_i$ in $\Z[\lambda_1, \dots, \lambda_8, \mu_1, \dots, \mu_8]$. In order to prove the theorem for all prime characteristics, we just have to check that the system of equations
\[
(f_i=0)_{1\leq i \leq 128}
\]
is still inconsistent for all infinite fields whose characteristic divides the natural number~$m$. 

Checking the result for all prime divisors of $m$ is not realistic, because those characteristics are too large for the computer algebra package we used. One way to proceed here is to choose a different monomial order when calculating a Gr\"obner basis. This gives us a way to write $1$ as a different linear combination of the $(f_i)_{1\leq i\leq 128}$, and thus obtain a new natural number $m'$. The problem is now reduced to a practical size by checking the result for the common prime roots of $m$ and $m'$.

Concretely, the number $m$ which arises in the proof of Theorem~\ref{Theorem degree 3} from the code in Remark~\ref{code} is $
m=\seqsplit{8702366371611141098671784574059986886835524773056180191922975621327712489528654597642065042289512}$, which factors into primes as 
\begin{align*}
m= 2^3 &\times 3^2\times 7\times 1049 \times 14479\times 12133021861\times 16113739806343\\
&\times 6887165068164869\times 18481555969123738547989\\
&\times 45682348205218520398213951.
\end{align*}
A different monomial order may be chosen in line 1 of the code in Remark~\ref{code}: for instance, swapping $\mu_7$ and $\mu_8$ is expressed as 
\begin{Verbatim}[numbers=left]
ring r=0,(x(1..8),y(1..6),y(8),y(7)),dp;
\end{Verbatim}
and leads to the number $
m'=\seqsplit{2594280368805968297556964255364362028715956773862814343963891594296990709562550944826763410}$, which factors into primes as 
\begin{align*}
m'= 2&\times 3\times 5\times 7\times 13\times 67\times 8878743659183\times 980432662345198665679\\ &\times 1629334190706617301312819947580437423912985358007543.
\end{align*}
Clearly, only $2$, $3$ and $7$ remain as prime characteristics to be checked by hand, which can easily be done in \textsc{Singular} by using
\begin{Verbatim}[numbers=left]
ring r=2,(x(1..8),y(1..8)),dp;
\end{Verbatim}
and its variations in line 1 of the code in Remark~\ref{code}.
\end{proof}

\begin{corollary}\label{Corollary subvariety prime}
Let $\K$ an infinite field of prime characteristic. Any variety of non-associative $\K$-algebras which is a \LACC\ category is a variety of quasi-Lie algebras.\noproof
\end{corollary}

\section{Subvarieties of \texorpdfstring{$\qLiek$}{qLieK}}\label{Section Subvarieties}
We continue working towards the main result of this article, now showing that the only varieties of quasi-Lie algebras over an infinite field which form non-abelian \LACC\ categories are~$\Liek$ and $\qLiek$.

\begin{proposition}\label{Proposition Lie}
Let $\K$ be an infinite field. Let $\V$ be a subvariety of either $\Liek$ or $\qLie_K$ determined by a collection of equations of degree $3$ or higher. If $\V$ is a \LACC\ category, then it is an abelian category.
\end{proposition}
\begin{proof}
Let $\V$ be a variety of non-associative $\K$-algebras which forms a non-abelian \LACC\ category. Then Corollary~\ref{Corollary subvariety} and Corollary~\ref{Corollary subvariety prime} together tell us that $\V$ is a subvariety of $\qLiek$. Using induction on the degree, we shall prove that~$\V$ does not admit any non-trivial identities, besides those obtained from the Jacobi identity and either $xx=0$ or $xy=-yx$. The results of Section~\ref{Section Commutative} already show that there are no additional homogeneous identities of degree $2$.

Let us now suppose that $\V$ satisfies a homogeneous identity of degree $3$. By Theorem~\ref{Theorem Linearity General} this identity may be written in the form
\[
b(xy)=\lambda (bx)y+\mu x(by).
\]
However, by the Jacobi identity we then have that 
\[
0=(\lambda-1)(bx)y+(\mu-1)x(by).
\]
Letting $B$, $X$ and $Y$ be free algebras as above, it is clear from Theorem~\ref{Theorem Homogeneity} that $\lambda =\mu=1$. We conclude that each algebra in $\V$ has a zero product---a contradiction with the assumption that $\V$ is a non-abelian category. Hence there are no additional homogeneous identities of degree $3$.

Let us next suppose that $\V$ satisfies a homogeneous identity of degree $n+1$, where $n\geq 3$, while all homogeneous identities of degree at most $n$ are induced by the Jacobi identity and either $xx=0$ or $xy=-yx$. By Theorem~\ref{Theorem Linearity General} we may assume that this identity is multilinear. We write it as 
\[
0=\psi(b,x_1,\dots,x_n),
\]
so that the polynomial on the right may be considered as an element of
\[
B+X_1+\cdots+X_n
\]
where $B$ is as above (the generator $b$ written as $b_i$ when several copies of $B$ are considered) and $X_i$ is the free algebra on a single generator $x_i$. Since $\V$ is a variety of quasi-Lie algebras, we may use anticommutativity and the Jacobi identity to rewrite the identity $0=\psi(b,x_1,\dots,x_n)$ in the shape
\[
0=\sum_{i=1}^n\lambda_i\varphi_i(x_1,\dots,bx_i,\dots,x_n)
\]
where $\varphi_i(x_1,\dots,b_ix_i,\dots,x_n)$ is in $B\flat X_1+\cdots+B\flat X_n$. By the assumption that the category $\V$ is \LACC, the canonical comparison
\[
B\flat X_1+\cdots+B\flat X_n\to B\flat (X_1+\cdots+X_n)
\]
is an isomorphism. As a consequence, we have that 
\[
0=\sum_{i=1}^n\lambda_i\varphi_i(x_1,\dots,b_ix_i,\dots,x_n)
\]
in $B\flat X_1+\cdots+B\flat X_n$. Theorem~\ref{Theorem Homogeneity} then implies that $\varphi_i(x_1,\dots,b_ix_i,\dots,x_n)=0$ for each $i\in\{1,\dots,n\}$. These, however, are equations in $B\flat X_1+\cdots+B\flat X_n$, hence in $B+ X_1+\cdots+B+X_n$, of a degree at most $n$, so that they must be induced by the Jacobi identity and either $xx=0$ or $xy=-yx$. 
\end{proof}

\begin{remark}
Although Theorem~\ref{Theorem Linearity General} is only stated for $\Liek$, it easy to verify that it also holds for $\qLiek$ when $\kar (\K) = 2$. The free algebra in $\qLiek$ is equal to the free algebra in $\Liek$ with the extra linearly independent components formed by the repetition of elements $xx$, but they do not play any role since
\[
y(xx) = (yx)x + x(yx) = (yx)x + (yx)x = 0.
\]
\end{remark}

\begin{theorem}\label{Theorem Lie}
Let $\K$ be an infinite field. The only varieties of non-associative $\K$-algebras that form non-abelian \LACC\ categories are
\begin{enumerate}
\item the variety $\Liek$, when $\kar(\K)\neq 2$;
\item the varieties $\Liek$ and $\qLiek$, when $\kar(\K)=2$.
\end{enumerate}
\end{theorem}
\begin{proof}
This is a combination of Proposition~\ref{Proposition Lie} with Corollary~\ref{Corollary subvariety}, Corollary~\ref{Corollary subvariety prime} and Theorem~\ref{Theorem JJAlg char 2}.
\end{proof}

\section{About \texorpdfstring{$n$}{n}-ary operations}\label{Section n-Ary}

Given any natural number $n$, an \defn{$n$-algebra $A$ over $\K$} is a $\K$-vector space equipped with an $n$-linear operation ${A^n\to A}$, so a linear map ${A^{\tensor n}\to A}$. We write $\nAlg_{\K}$ for the category of $n$-algebras over~$\K$, with linear maps between them that preserve the operation. A \defn{variety of $n$-algebras} is any equationally defined class of such, considered as a full subcategory~$\V$ of~$\nAlg_{\K}$. 

We now study what happens when instead of considering a variety of non-associative algebras (=~$2$-algebras) we take a variety of $n$-algebras for arbitrary~$n$. When $n=0$ or $n=1$, this is not very interesting:

\begin{proposition}\label{Proposition Nullary Operation}
A variety of $0$-algebras is always a \LACC\ category, but never an abelian or even a semi-abelian category.
\end{proposition}
\begin{proof}
If $n=0$, then $\nAlg_\K\cong(\K\downarrow \Vect_\K)$ because $A^{\tensor 0}=\K$: objects are linear maps $\K\to A$, morphisms are commutative triangles between those. By Lemma~2.3 in~\cite{Bourn-Gray}, this coslice category of $\Vect_\K$ is still \LACC, even though it is no longer a (semi-)abelian category, since it is not pointed. 
\end{proof}

\begin{proposition}\label{Proposition Unary Operation}
Any variety of $1$-algebras is an abelian category.
\end{proposition}
\begin{proof}
If $n=1$, then an object in $\nAlg_\K$ is a vector space $A$ equipped with an endomorphism $t_A\colon{A\to A}$, and a morphism $f\colon (A,t_A)\to (B,t_B)$ is a linear map $f\colon A\to B$ such that $f\comp t_A=t_B\comp f$. Hence $\nAlg_\K\cong\Fun(\N,\Vect_\K)$, the category of functors from the monoid of natural numbers (considered as a one-object category) to $\Vect_\K$. Here a functor $T\colon{\N\to \Vect_\K}$ sends the unique object $*$ of $\N$ to a vector space $A=T(*)$, sends $0$ to $1_A$, $1$ to an endomorphism $t_A\colon{A\to A}$, and $m\geq 2$ to $t^m_A\colon{A\to A}$. As a consequence, $\nAlg_\K$ is an abelian category, being a category of presheaves over $\Vect_\K$. The claim now follows, since any subvariety of a variety of algebras which is an abelian category is again an abelian category.
\end{proof}

When $n\geq 3$, the main obstruction we encounter is the apparent absence in the literature of an analogue of Theorem~\ref{Theorem Homogeneity} for $n$-ary operations. Without entering in too much detail, the idea of the proof for $n=2$ is that if we have an identity $f = f_0 + \cdots + f_k$ with $f_j$ being the sum of the monomials where $x_1$ has degree $j$, then both $\lambda^j f(x_1, \dots, x_r)=0$ and $f(\lambda x_1, \dots, x_r)=0$ are identities, so, if we do not run out of variables $\lambda$ to choose from---which cannot happen when $\K$ is infinite, then $\K$-linearity will imply the result. Therefore, extending Theorem~\ref{Theorem Homogeneity} to $n$-ary operations is a straightforward exercise. This gives us:

\begin{theorem}\label{Theorem n-Ary}
Let $\K$ be an infinite field and let $\V$ be a variety of $n$-algebras where $n \geq 3$. Then the comparison map
\[
B\flat X_1+\cdots+B\flat X_{n-1}\to B\flat (X_1+\cdots+X_{n-1})
\]
is surjective only when $\V$ is an abelian category.
\end{theorem}

\begin{proof}
We write the given $n$-linear operation as a bracket $[-,\dots,-]$. By the $n$-ary version of Theorem~\ref{Theorem Homogeneity}, there is no expression in $B\flat X_1+\cdots+B\flat X_{n-1}$ that can be mapped to the element $[b, x_1, \dots, x_{n-1}] \in B\flat (X_1+\cdots+X_{n-1})$---unless it is zero.
\end{proof}

A variety of algebras which forms a semi-abelian category~\cite{JaMaTh} is an \defn{algebraically coherent} category in the sense of~\cite{acc} when for all algebras $B$, $X$, $Y$ the comparison morphism
\[
B\flat X+B\flat Y\to B\flat (X+Y)
\]
is a surjective algebra morphism. This means that we may interpret the above result as follows.

\begin{corollary}
If $\V$ is a variety of $n$-algebras where $n \geq 3$, then $\V$ can only be an algebraically coherent category if it is an abelian category.\noproof
\end{corollary}

Theorem~\ref{Theorem Lie}, Proposition~\ref{Proposition Nullary Operation}, Proposition~\ref{Proposition Unary Operation} and Theorem~\ref{Theorem n-Ary} together now allow us to prove our main theorem.

\begin{theorem}
Let $\K$ be an infinite field. Let $\V$ be a variety of $n$-algebras over $\K$ which is a non-abelian \emph{locally algebraically cartesian closed} category. Then $n = 2$ and
\begin{enumerate}
\item if $\kar(\K)\neq 2$, then $\V = \Liek = \qLiek$;
\item if $\kar(\K)= 2$, then $\V = \Liek$ or $\V = \qLiek$.\noproof
\end{enumerate}
\end{theorem}

\section{Final remarks}\label{Section Final Remarks}

\subsection{Related work}
Similar applications of computational commutative algebra to algebraic operads appear in~\cite{MR3642294, MR3693146, MR3640086, Graces}.

\subsection{A different proof technique?}
Theorem~\ref{Theorem degree 3} makes us wonder whether a less computationally involved method may be found for its proof. We did not manage to reduce the system of equations in Appendix~\ref{Appendix} to a smaller one, but perhaps a totally different argument can be discovered. This may also help extending the result to finite fields---if such an extension is possible at all.

\subsection{The case of sets}
Besides the variety of Lie algebras, the only varieties of algebras currently known to be \emph{locally algebraically cartesian closed} categories are varieties of groups. This raises the question whether perhaps the result of this paper has a cartesian version. This would first of all mean characterising the varieties of pointed magmas---whose objects are sets with a given base-point and a binary operation that preserves it---which are \LACC\ categories. An answer to this question would either distinguish groups amongst pointed magmas via a categorical condition, or provide new examples of \LACC\ categories. 

\subsection{Several operations}
The problem investigated in this article becomes very different when instead of a single bilinear operation, several bilinear or $n$-linear algebraic or coalgebraic operations on a given object are considered. We do not know under which conditions adding operations to a variety gives a way of constructing new~\LACC\ categories.

\subsection{Algebraically exponentiable objects and morphisms}
We studied algebraic exponentiation from a global perspective, by answering the question in which kind of varieties of $n$-algebras \emph{all} endofunctors $B\flat (-)$ are left adjoints. As explained in the introduction, the left adjointness of $B\flat (-)$ can also be seen as a property of a single given object~$B$, in which case it is said to be \defn{algebraically exponentiable}. More generally, algebraically exponentiable morphisms may be considered~\cite{Gray2012, Bourn-Gray}. Characterising those objects and morphisms may turn out to be an interesting challenge.

\section*{Acknowledgements}
We wish to thank Pierre-Alain Jacqmin for an important remark on the prime characteristic case. We also thank the referees for their helpful comments and suggestions.

The authors acknowledge the use of the Computer Cluster of the \emph{Centro de Supercomputaci\'on de Galicia} (CESGA).
The first author would like to thank the \emph{Institut de Recherche en Mathématique et Physique} (IRMP) for its kind hospitality during his stays in Louvain-la-Neuve, and the \emph{Department of Mathematics} of the University of Santiago.

\appendix
\section{The coefficients \texorpdfstring{$f_{i}$}{fi} used in the proof of Theorem~\ref{Theorem degree 3}}\label{Appendix}
\allowdisplaybreaks
\tiny
\begin{align*}
f_{1} &= \mu_{5}\mu_{5} + \mu_{6}\mu_{1} + \mu_{7}\lambda_{5} + \mu_{8}\lambda_{1} - 1 
& f_{17} &= \lambda_{5}\mu_{5} + \lambda_{6}\mu_{1} + \lambda_{7}\lambda_{5} + \lambda_{8}\lambda_{1} \\
f_{2} &= \mu_{5}\mu_{6} + \mu_{6}\mu_{2} + \mu_{7}\lambda_{6} + \mu_{8}\lambda_{2} 
&f_{18} &= \lambda_{5}\mu_{6} + \lambda_{6}\mu_{2} + \lambda_{7}\lambda_{6} + \lambda_{8}\lambda_{2} \\
f_{3} &= \mu_{5}\mu_{7} + \mu_{6}\mu_{3} + \mu_{7}\lambda_{7} + \mu_{8}\lambda_{3} 
&f_{19} &= \lambda_{5}\mu_{7} + \lambda_{6}\mu_{3} + \lambda_{7}\lambda_{7} + \lambda_{8}\lambda_{3} - 1 \\
f_{4} &= \mu_{5}\mu_{8} + \mu_{6}\mu_{4} + \mu_{7}\lambda_{8} + \mu_{8}\lambda_{4}
&f_{20} &= \lambda_{5}\mu_{8} + \lambda_{6}\mu_{4} + \lambda_{7}\lambda_{8} + \lambda_{8}\lambda_{4} \\
f_{5} &= \mu_{5}\mu_{1} + \mu_{6}\mu_{5} + \mu_{7}\lambda_{1} + \mu_{8}\lambda_{5} + \mu_{2} 
&f_{21} &= \lambda_{5}\mu_{1} + \lambda_{6}\mu_{5} + \lambda_{7}\lambda_{1} + \lambda_{8}\lambda_{5} + \lambda_{2} \\
f_{6} &= \mu_{5}\mu_{2} + \mu_{6}\mu_{6} + \mu_{7}\lambda_{2} + \mu_{8}\lambda_{6} + \mu_{1}
& f_{22} &= \lambda_{5}\mu_{2} + \lambda_{6}\mu_{6} + \lambda_{7}\lambda_{2} + \lambda_{8}\lambda_{6} + \lambda_{1} \\
f_{7} &= \mu_{5}\mu_{3} + \mu_{6}\mu_{7} + \mu_{7}\lambda_{3} + \mu_{8}\lambda_{7} + \mu_{4}
&f_{23} &= \lambda_{5}\mu_{3} + \lambda_{6}\mu_{7} + \lambda_{7}\lambda_{3} + \lambda_{8}\lambda_{7} + \lambda_{4} \\
f_{8} &= \mu_{5}\mu_{4} + \mu_{6}\mu_{8} + \mu_{7}\lambda_{4} + \mu_{8}\lambda_{8} + \mu_{3} 
&f_{24} &= \lambda_{5}\mu_{4} + \lambda_{6}\mu_{8} + \lambda_{7}\lambda_{4} + \lambda_{8}\lambda_{8} + \lambda_{3} \\
f_{9} &= \mu_{1}\mu_{5} + \mu_{2}\mu_{1} + \mu_{3}\lambda_{5} + \mu_{4}\lambda_{1} 
&f_{25} &= \lambda_{1}\mu_{5} + \lambda_{2}\mu_{1} + \lambda_{3}\lambda_{5} + \lambda_{4}\lambda_{1} \\
f_{10} &= \mu_{1}\mu_{6} + \mu_{2}\mu_{2} + \mu_{3}\lambda_{6} + \mu_{4}\lambda_{2} - 1
&f_{26} &= \lambda_{1}\mu_{6} + \lambda_{2}\mu_{2} + \lambda_{3}\lambda_{6} + \lambda_{4}\lambda_{2} \\
f_{11} &= \mu_{1}\mu_{7} + \mu_{2}\mu_{3} + \mu_{3}\lambda_{7} + \mu_{4}\lambda_{3}
&f_{27} &= \lambda_{1}\mu_{7} + \lambda_{2}\mu_{3} + \lambda_{3}\lambda_{7} + \lambda_{4}\lambda_{3} \\
f_{12} &= \mu_{1}\mu_{8} + \mu_{2}\mu_{4} + \mu_{3}\lambda_{8} + \mu_{4}\lambda_{4}
&f_{28} &= \lambda_{1}\mu_{8} + \lambda_{2}\mu_{4} + \lambda_{3}\lambda_{8} + \lambda_{4}\lambda_{4} - 1 \\
f_{13} &= \mu_{1}\mu_{1} + \mu_{2}\mu_{5} + \mu_{3}\lambda_{1} + \mu_{4}\lambda_{5} + \mu_{6}
&f_{29} &= \lambda_{1}\mu_{1} + \lambda_{2}\mu_{5} + \lambda_{3}\lambda_{1} + \lambda_{4}\lambda_{5} + \lambda_{6} \\
f_{14} &= \mu_{1}\mu_{2} + \mu_{2}\mu_{6} + \mu_{3}\lambda_{2} + \mu_{4}\lambda_{6} + \mu_{5}
&f_{30} &= \lambda_{1}\mu_{2} + \lambda_{2}\mu_{6} + \lambda_{3}\lambda_{2} + \lambda_{4}\lambda_{6} + \lambda_{5} \\
f_{15} &= \mu_{1}\mu_{3} + \mu_{2}\mu_{7} + \mu_{3}\lambda_{3} + \mu_{4}\lambda_{7} + \mu_{8}
&f_{31} &= \lambda_{1}\mu_{3} + \lambda_{2}\mu_{7} + \lambda_{3}\lambda_{3} + \lambda_{4}\lambda_{7} + \lambda_{8} \\
f_{16} &= \mu_{1}\mu_{4} + \mu_{2}\mu_{8} + \mu_{3}\lambda_{4} + \mu_{4}\lambda_{8} + \mu_{7}
&f_{32} &= \lambda_{1}\mu_{4} + \lambda_{2}\mu_{8} + \lambda_{3}\lambda_{4} + \lambda_{4}\lambda_{8} + \lambda_{7}
\end{align*}
\begin{align*}
f_{33} = \mu_{1}\mu_{1}\mu_{5} &+ \mu_{2}\lambda_{1}\mu_{5} + \mu_{3}\mu_{1}\lambda_{5} + \mu_{4}\lambda_{1}\lambda_{5} + \mu_{1}\mu_{3}\mu_{1} + \mu_{2}\lambda_{3}\mu_{1} + \mu_{3}\mu_{3}\lambda_{1} + \mu_{4}\lambda_{3}\lambda_{1} \\ &+ \mu_{5}\mu_{5}\mu_{7} + \mu_{6}\lambda_{5}\mu_{7} + \mu_{7}\mu_{5}\lambda_{7} + \mu_{8}\lambda_{5}\lambda_{7} + \mu_{5}\mu_{7}\mu_{3} + \mu_{6}\lambda_{7}\mu_{3} + \mu_{7}\mu_{7}\lambda_{3} + \mu_{8}\lambda_{7}\lambda_{3} \\
f_{34} = \mu_{1}\mu_{1}\mu_{6} &+ \mu_{2}\lambda_{1}\mu_{6} + \mu_{3}\mu_{1}\lambda_{6} + \mu_{4}\lambda_{1}\lambda_{6} + \mu_{1}\mu_{3}\mu_{2} + \mu_{2}\lambda_{3}\mu_{2} + \mu_{3}\mu_{3}\lambda_{2} + \mu_{4}\lambda_{3}\lambda_{2} \\ &+ \mu_{5}\mu_{6}\mu_{7} + \mu_{6}\lambda_{6}\mu_{7} + \mu_{7}\mu_{6}\lambda_{7} + \mu_{8}\lambda_{6}\lambda_{7} + \mu_{5}\mu_{8}\mu_{3} + \mu_{6}\lambda_{8}\mu_{3} + \mu_{7}\mu_{8}\lambda_{3} + \mu_{8}\lambda_{8}\lambda_{3} \\
f_{35} = \mu_{1}\mu_{1}\mu_{7} &+ \mu_{2}\lambda_{1}\mu_{7} + \mu_{3}\mu_{1}\lambda_{7} + \mu_{4}\lambda_{1}\lambda_{7} + \mu_{1}\mu_{3}\mu_{3} + \mu_{2}\lambda_{3}\mu_{3} + \mu_{3}\mu_{3}\lambda_{3} + \mu_{4}\lambda_{3}\lambda_{3} \\ &+ \mu_{5}\mu_{5}\mu_{5} + \mu_{6}\lambda_{5}\mu_{5} + \mu_{7}\mu_{5}\lambda_{5} + \mu_{8}\lambda_{5}\lambda_{5} + \mu_{5}\mu_{7}\mu_{1} + \mu_{6}\lambda_{7}\mu_{1} + \mu_{7}\mu_{7}\lambda_{1} + \mu_{8}\lambda_{7}\lambda_{1} \\
f_{36} = \mu_{1}\mu_{1}\mu_{8} &+ \mu_{2}\lambda_{1}\mu_{8} + \mu_{3}\mu_{1}\lambda_{8} + \mu_{4}\lambda_{1}\lambda_{8} + \mu_{1}\mu_{3}\mu_{4} + \mu_{2}\lambda_{3}\mu_{4} + \mu_{3}\mu_{3}\lambda_{4} + \mu_{4}\lambda_{3}\lambda_{4} \\ &+ \mu_{5}\mu_{6}\mu_{5} + \mu_{6}\lambda_{6}\mu_{5} + \mu_{7}\mu_{6}\lambda_{5} + \mu_{8}\lambda_{6}\lambda_{5} + \mu_{5}\mu_{8}\mu_{1} + \mu_{6}\lambda_{8}\mu_{1} + \mu_{7}\mu_{8}\lambda_{1} + \mu_{8}\lambda_{8}\lambda_{1} \\
f_{37} = \mu_{1}\mu_{2}\mu_{5} &+ \mu_{2}\lambda_{2}\mu_{5} + \mu_{3}\mu_{2}\lambda_{5} + \mu_{4}\lambda_{2}\lambda_{5} + \mu_{1}\mu_{4}\mu_{1} + \mu_{2}\lambda_{4}\mu_{1} + \mu_{3}\mu_{4}\lambda_{1} + \mu_{4}\lambda_{4}\lambda_{1} \\ &+ \mu_{5}\mu_{5}\mu_{8} + \mu_{6}\lambda_{5}\mu_{8} + \mu_{7}\mu_{5}\lambda_{8} + \mu_{8}\lambda_{5}\lambda_{8} + \mu_{5}\mu_{7}\mu_{4} + \mu_{6}\lambda_{7}\mu_{4} + \mu_{7}\mu_{7}\lambda_{4} + \mu_{8}\lambda_{7}\lambda_{4} \\
f_{38} = \mu_{1}\mu_{2}\mu_{6} &+ \mu_{2}\lambda_{2}\mu_{6} + \mu_{3}\mu_{2}\lambda_{6} + \mu_{4}\lambda_{2}\lambda_{6} + \mu_{1}\mu_{4}\mu_{2} + \mu_{2}\lambda_{4}\mu_{2} + \mu_{3}\mu_{4}\lambda_{2} + \mu_{4}\lambda_{4}\lambda_{2} \\ &+ \mu_{5}\mu_{6}\mu_{8} + \mu_{6}\lambda_{6}\mu_{8} + \mu_{7}\mu_{6}\lambda_{8} + \mu_{8}\lambda_{6}\lambda_{8} + \mu_{5}\mu_{8}\mu_{4} + \mu_{6}\lambda_{8}\mu_{4} + \mu_{7}\mu_{8}\lambda_{4} + \mu_{8}\lambda_{8}\lambda_{4} \\
f_{39} = \mu_{1}\mu_{2}\mu_{7} &+ \mu_{2}\lambda_{2}\mu_{7} + \mu_{3}\mu_{2}\lambda_{7} + \mu_{4}\lambda_{2}\lambda_{7} + \mu_{1}\mu_{4}\mu_{3} + \mu_{2}\lambda_{4}\mu_{3} + \mu_{3}\mu_{4}\lambda_{3} + \mu_{4}\lambda_{4}\lambda_{3} \\ &+ \mu_{5}\mu_{5}\mu_{6} + \mu_{6}\lambda_{5}\mu_{6} + \mu_{7}\mu_{5}\lambda_{6} + \mu_{8}\lambda_{5}\lambda_{6} + \mu_{5}\mu_{7}\mu_{2} + \mu_{6}\lambda_{7}\mu_{2} + \mu_{7}\mu_{7}\lambda_{2} + \mu_{8}\lambda_{7}\lambda_{2} \\
f_{40} = \mu_{1}\mu_{2}\mu_{8} &+ \mu_{2}\lambda_{2}\mu_{8} + \mu_{3}\mu_{2}\lambda_{8} + \mu_{4}\lambda_{2}\lambda_{8} + \mu_{1}\mu_{4}\mu_{4} + \mu_{2}\lambda_{4}\mu_{4} + \mu_{3}\mu_{4}\lambda_{4} + \mu_{4}\lambda_{4}\lambda_{4} \\ &+ \mu_{5}\mu_{6}\mu_{6} + \mu_{6}\lambda_{6}\mu_{6} + \mu_{7}\mu_{6}\lambda_{6} + \mu_{8}\lambda_{6}\lambda_{6} + \mu_{5}\mu_{8}\mu_{2} + \mu_{6}\lambda_{8}\mu_{2} + \mu_{7}\mu_{8}\lambda_{2} + \mu_{8}\lambda_{8}\lambda_{2} \\
f_{41} = \mu_{1}\mu_{5}\mu_{5} &+ \mu_{2}\lambda_{5}\mu_{5} + \mu_{3}\mu_{5}\lambda_{5} + \mu_{4}\lambda_{5}\lambda_{5} + \mu_{1}\mu_{7}\mu_{1} + \mu_{2}\lambda_{7}\mu_{1} + \mu_{3}\mu_{7}\lambda_{1} + \mu_{4}\lambda_{7}\lambda_{1} \\ &+ \mu_{5}\mu_{1}\mu_{7} + \mu_{6}\lambda_{1}\mu_{7} + \mu_{7}\mu_{1}\lambda_{7} + \mu_{8}\lambda_{1}\lambda_{7} + \mu_{5}\mu_{3}\mu_{3} + \mu_{6}\lambda_{3}\mu_{3} + \mu_{7}\mu_{3}\lambda_{3} + \mu_{8}\lambda_{3}\lambda_{3} \\
f_{42} = \mu_{1}\mu_{5}\mu_{6} &+ \mu_{2}\lambda_{5}\mu_{6} + \mu_{3}\mu_{5}\lambda_{6} + \mu_{4}\lambda_{5}\lambda_{6} + \mu_{1}\mu_{7}\mu_{2} + \mu_{2}\lambda_{7}\mu_{2} + \mu_{3}\mu_{7}\lambda_{2} + \mu_{4}\lambda_{7}\lambda_{2} \\ &+ \mu_{5}\mu_{2}\mu_{7} + \mu_{6}\lambda_{2}\mu_{7} + \mu_{7}\mu_{2}\lambda_{7} + \mu_{8}\lambda_{2}\lambda_{7} + \mu_{5}\mu_{4}\mu_{3} + \mu_{6}\lambda_{4}\mu_{3} + \mu_{7}\mu_{4}\lambda_{3} + \mu_{8}\lambda_{4}\lambda_{3} \\
f_{43} = \mu_{1}\mu_{5}\mu_{7} &+ \mu_{2}\lambda_{5}\mu_{7} + \mu_{3}\mu_{5}\lambda_{7} + \mu_{4}\lambda_{5}\lambda_{7} + \mu_{1}\mu_{7}\mu_{3} + \mu_{2}\lambda_{7}\mu_{3} + \mu_{3}\mu_{7}\lambda_{3} + \mu_{4}\lambda_{7}\lambda_{3} \\ &+ \mu_{5}\mu_{1}\mu_{5} + \mu_{6}\lambda_{1}\mu_{5} + \mu_{7}\mu_{1}\lambda_{5} + \mu_{8}\lambda_{1}\lambda_{5} + \mu_{5}\mu_{3}\mu_{1} + \mu_{6}\lambda_{3}\mu_{1} + \mu_{7}\mu_{3}\lambda_{1} + \mu_{8}\lambda_{3}\lambda_{1} \\
f_{44} = \mu_{1}\mu_{5}\mu_{8} &+ \mu_{2}\lambda_{5}\mu_{8} + \mu_{3}\mu_{5}\lambda_{8} + \mu_{4}\lambda_{5}\lambda_{8} + \mu_{1}\mu_{7}\mu_{4} + \mu_{2}\lambda_{7}\mu_{4} + \mu_{3}\mu_{7}\lambda_{4} + \mu_{4}\lambda_{7}\lambda_{4} \\ &+ \mu_{5}\mu_{2}\mu_{5} + \mu_{6}\lambda_{2}\mu_{5} + \mu_{7}\mu_{2}\lambda_{5} + \mu_{8}\lambda_{2}\lambda_{5} + \mu_{5}\mu_{4}\mu_{1} + \mu_{6}\lambda_{4}\mu_{1} + \mu_{7}\mu_{4}\lambda_{1} + \mu_{8}\lambda_{4}\lambda_{1} \\
f_{45} = \mu_{1}\mu_{6}\mu_{5} &+ \mu_{2}\lambda_{6}\mu_{5} + \mu_{3}\mu_{6}\lambda_{5} + \mu_{4}\lambda_{6}\lambda_{5} + \mu_{1}\mu_{8}\mu_{1} + \mu_{2}\lambda_{8}\mu_{1} + \mu_{3}\mu_{8}\lambda_{1} + \mu_{4}\lambda_{8}\lambda_{1} \\ &+ \mu_{5}\mu_{1}\mu_{8} + \mu_{6}\lambda_{1}\mu_{8} + \mu_{7}\mu_{1}\lambda_{8} + \mu_{8}\lambda_{1}\lambda_{8} + \mu_{5}\mu_{3}\mu_{4} + \mu_{6}\lambda_{3}\mu_{4} + \mu_{7}\mu_{3}\lambda_{4} + \mu_{8}\lambda_{3}\lambda_{4} \\
f_{46} = \mu_{1}\mu_{6}\mu_{6} &+ \mu_{2}\lambda_{6}\mu_{6} + \mu_{3}\mu_{6}\lambda_{6} + \mu_{4}\lambda_{6}\lambda_{6} + \mu_{1}\mu_{8}\mu_{2} + \mu_{2}\lambda_{8}\mu_{2} + \mu_{3}\mu_{8}\lambda_{2} + \mu_{4}\lambda_{8}\lambda_{2} \\ &+ \mu_{5}\mu_{2}\mu_{8} + \mu_{6}\lambda_{2}\mu_{8} + \mu_{7}\mu_{2}\lambda_{8} + \mu_{8}\lambda_{2}\lambda_{8} + \mu_{5}\mu_{4}\mu_{4} + \mu_{6}\lambda_{4}\mu_{4} + \mu_{7}\mu_{4}\lambda_{4} + \mu_{8}\lambda_{4}\lambda_{4} \\
f_{47} = \mu_{1}\mu_{6}\mu_{7} &+ \mu_{2}\lambda_{6}\mu_{7} + \mu_{3}\mu_{6}\lambda_{7} + \mu_{4}\lambda_{6}\lambda_{7} + \mu_{1}\mu_{8}\mu_{3} + \mu_{2}\lambda_{8}\mu_{3} + \mu_{3}\mu_{8}\lambda_{3} + \mu_{4}\lambda_{8}\lambda_{3} \\ &+ \mu_{5}\mu_{1}\mu_{6} + \mu_{6}\lambda_{1}\mu_{6} + \mu_{7}\mu_{1}\lambda_{6} + \mu_{8}\lambda_{1}\lambda_{6} + \mu_{5}\mu_{3}\mu_{2} + \mu_{6}\lambda_{3}\mu_{2} + \mu_{7}\mu_{3}\lambda_{2} + \mu_{8}\lambda_{3}\lambda_{2} \\
f_{48} = \mu_{1}\mu_{6}\mu_{8} &+ \mu_{2}\lambda_{6}\mu_{8} + \mu_{3}\mu_{6}\lambda_{8} + \mu_{4}\lambda_{6}\lambda_{8} + \mu_{1}\mu_{8}\mu_{4} + \mu_{2}\lambda_{8}\mu_{4} + \mu_{3}\mu_{8}\lambda_{4} + \mu_{4}\lambda_{8}\lambda_{4} \\ &+ \mu_{5}\mu_{2}\mu_{6} + \mu_{6}\lambda_{2}\mu_{6} + \mu_{7}\mu_{2}\lambda_{6} + \mu_{8}\lambda_{2}\lambda_{6} + \mu_{5}\mu_{4}\mu_{2} + \mu_{6}\lambda_{4}\mu_{2} + \mu_{7}\mu_{4}\lambda_{2} + \mu_{8}\lambda_{4}\lambda_{2} 
\end{align*}
\begin{align*}
f_{49} &= - \lambda_{1}\mu_{1} - \lambda_{2}\lambda_{1} + \mu_{1}\mu_{2}\mu_{2} + \mu_{2}\lambda_{2}\mu_{2} + \mu_{3}\mu_{2}\lambda_{2} + \mu_{4}\lambda_{2}\lambda_{2} + \mu_{1}\mu_{4}\mu_{6} + \mu_{2}\lambda_{4}\mu_{6} + \mu_{3}\mu_{4}\lambda_{6} + \mu_{4}\lambda_{4}\lambda_{6} \\
f_{50} &= - \lambda_{1}\mu_{2} - \lambda_{2}\lambda_{2} + \mu_{1}\mu_{1}\mu_{2} + \mu_{2}\lambda_{1}\mu_{2} + \mu_{3}\mu_{1}\lambda_{2} + \mu_{4}\lambda_{1}\lambda_{2} + \mu_{1}\mu_{3}\mu_{6} + \mu_{2}\lambda_{3}\mu_{6} + \mu_{3}\mu_{3}\lambda_{6} + \mu_{4}\lambda_{3}\lambda_{6} \\
f_{51} &= - \lambda_{1}\mu_{3} - \lambda_{2}\lambda_{3} + \mu_{1}\mu_{2}\mu_{1} + \mu_{2}\lambda_{2}\mu_{1} + \mu_{3}\mu_{2}\lambda_{1} + \mu_{4}\lambda_{2}\lambda_{1} + \mu_{1}\mu_{4}\mu_{5} + \mu_{2}\lambda_{4}\mu_{5} + \mu_{3}\mu_{4}\lambda_{5} + \mu_{4}\lambda_{4}\lambda_{5} \\
f_{52} &= - \lambda_{1}\mu_{4} - \lambda_{2}\lambda_{4} + \mu_{1}\mu_{1}\mu_{1} + \mu_{2}\lambda_{1}\mu_{1} + \mu_{3}\mu_{1}\lambda_{1} + \mu_{4}\lambda_{1}\lambda_{1} + \mu_{1}\mu_{3}\mu_{5} + \mu_{2}\lambda_{3}\mu_{5} + \mu_{3}\mu_{3}\lambda_{5} + \mu_{4}\lambda_{3}\lambda_{5} \\
f_{53} &= - \lambda_{1}\mu_{5} - \lambda_{2}\lambda_{5} + \mu_{5}\mu_{2}\mu_{2} + \mu_{6}\lambda_{2}\mu_{2} + \mu_{7}\mu_{2}\lambda_{2} + \mu_{8}\lambda_{2}\lambda_{2} + \mu_{5}\mu_{4}\mu_{6} + \mu_{6}\lambda_{4}\mu_{6} + \mu_{7}\mu_{4}\lambda_{6} + \mu_{8}\lambda_{4}\lambda_{6} \\
f_{54} &= - \lambda_{1}\mu_{6} - \lambda_{2}\lambda_{6} + \mu_{5}\mu_{1}\mu_{2} + \mu_{6}\lambda_{1}\mu_{2} + \mu_{7}\mu_{1}\lambda_{2} + \mu_{8}\lambda_{1}\lambda_{2} + \mu_{5}\mu_{3}\mu_{6} + \mu_{6}\lambda_{3}\mu_{6} + \mu_{7}\mu_{3}\lambda_{6} + \mu_{8}\lambda_{3}\lambda_{6} \\
f_{55} &= - \lambda_{1}\mu_{7} - \lambda_{2}\lambda_{7} + \mu_{5}\mu_{2}\mu_{1} + \mu_{6}\lambda_{2}\mu_{1} + \mu_{7}\mu_{2}\lambda_{1} + \mu_{8}\lambda_{2}\lambda_{1} + \mu_{5}\mu_{4}\mu_{5} + \mu_{6}\lambda_{4}\mu_{5} + \mu_{7}\mu_{4}\lambda_{5} + \mu_{8}\lambda_{4}\lambda_{5} \\
f_{56} &= - \lambda_{1}\mu_{8} - \lambda_{2}\lambda_{8} + \mu_{5}\mu_{1}\mu_{1} + \mu_{6}\lambda_{1}\mu_{1} + \mu_{7}\mu_{1}\lambda_{1} + \mu_{8}\lambda_{1}\lambda_{1} + \mu_{5}\mu_{3}\mu_{5} + \mu_{6}\lambda_{3}\mu_{5} + \mu_{7}\mu_{3}\lambda_{5} + \mu_{8}\lambda_{3}\lambda_{5} \\
f_{57} &= - \lambda_{3}\mu_{1} - \lambda_{4}\lambda_{1} + \mu_{1}\mu_{2}\mu_{4} + \mu_{2}\lambda_{2}\mu_{4} + \mu_{3}\mu_{2}\lambda_{4} + \mu_{4}\lambda_{2}\lambda_{4} + \mu_{1}\mu_{4}\mu_{8} + \mu_{2}\lambda_{4}\mu_{8} + \mu_{3}\mu_{4}\lambda_{8} + \mu_{4}\lambda_{4}\lambda_{8} \\
f_{58} &= - \lambda_{3}\mu_{2} - \lambda_{4}\lambda_{2} + \mu_{1}\mu_{1}\mu_{4} + \mu_{2}\lambda_{1}\mu_{4} + \mu_{3}\mu_{1}\lambda_{4} + \mu_{4}\lambda_{1}\lambda_{4} + \mu_{1}\mu_{3}\mu_{8} + \mu_{2}\lambda_{3}\mu_{8} + \mu_{3}\mu_{3}\lambda_{8} + \mu_{4}\lambda_{3}\lambda_{8} \\
f_{59} &= - \lambda_{3}\mu_{3} - \lambda_{4}\lambda_{3} + \mu_{1}\mu_{2}\mu_{3} + \mu_{2}\lambda_{2}\mu_{3} + \mu_{3}\mu_{2}\lambda_{3} + \mu_{4}\lambda_{2}\lambda_{3} + \mu_{1}\mu_{4}\mu_{7} + \mu_{2}\lambda_{4}\mu_{7} + \mu_{3}\mu_{4}\lambda_{7} + \mu_{4}\lambda_{4}\lambda_{7} \\
f_{60} &= - \lambda_{3}\mu_{4} - \lambda_{4}\lambda_{4} + \mu_{1}\mu_{1}\mu_{3} + \mu_{2}\lambda_{1}\mu_{3} + \mu_{3}\mu_{1}\lambda_{3} + \mu_{4}\lambda_{1}\lambda_{3} + \mu_{1}\mu_{3}\mu_{7} + \mu_{2}\lambda_{3}\mu_{7} + \mu_{3}\mu_{3}\lambda_{7} + \mu_{4}\lambda_{3}\lambda_{7} \\
f_{61} &= - \lambda_{3}\mu_{5} - \lambda_{4}\lambda_{5} + \mu_{5}\mu_{2}\mu_{4} + \mu_{6}\lambda_{2}\mu_{4} + \mu_{7}\mu_{2}\lambda_{4} + \mu_{8}\lambda_{2}\lambda_{4} + \mu_{5}\mu_{4}\mu_{8} + \mu_{6}\lambda_{4}\mu_{8} + \mu_{7}\mu_{4}\lambda_{8} + \mu_{8}\lambda_{4}\lambda_{8} \\
f_{62} &= - \lambda_{3}\mu_{6} - \lambda_{4}\lambda_{6} + \mu_{5}\mu_{1}\mu_{4} + \mu_{6}\lambda_{1}\mu_{4} + \mu_{7}\mu_{1}\lambda_{4} + \mu_{8}\lambda_{1}\lambda_{4} + \mu_{5}\mu_{3}\mu_{8} + \mu_{6}\lambda_{3}\mu_{8} + \mu_{7}\mu_{3}\lambda_{8} + \mu_{8}\lambda_{3}\lambda_{8} \\
f_{63} &= - \lambda_{3}\mu_{7} - \lambda_{4}\lambda_{7} + \mu_{5}\mu_{2}\mu_{3} + \mu_{6}\lambda_{2}\mu_{3} + \mu_{7}\mu_{2}\lambda_{3} + \mu_{8}\lambda_{2}\lambda_{3} + \mu_{5}\mu_{4}\mu_{7} + \mu_{6}\lambda_{4}\mu_{7} + \mu_{7}\mu_{4}\lambda_{7} + \mu_{8}\lambda_{4}\lambda_{7} \\
f_{64} &= - \lambda_{3}\mu_{8} - \lambda_{4}\lambda_{8} + \mu_{5}\mu_{1}\mu_{3} + \mu_{6}\lambda_{1}\mu_{3} + \mu_{7}\mu_{1}\lambda_{3} + \mu_{8}\lambda_{1}\lambda_{3} + \mu_{5}\mu_{3}\mu_{7} + \mu_{6}\lambda_{3}\mu_{7} + \mu_{7}\mu_{3}\lambda_{7} + \mu_{8}\lambda_{3}\lambda_{7} \\
f_{65} &= - \lambda_{5}\mu_{1} - \lambda_{6}\lambda_{1} + \mu_{1}\mu_{6}\mu_{2} + \mu_{2}\lambda_{6}\mu_{2} + \mu_{3}\mu_{6}\lambda_{2} + \mu_{4}\lambda_{6}\lambda_{2} + \mu_{1}\mu_{8}\mu_{6} + \mu_{2}\lambda_{8}\mu_{6} + \mu_{3}\mu_{8}\lambda_{6} + \mu_{4}\lambda_{8}\lambda_{6} \\
f_{66} &= - \lambda_{5}\mu_{2} - \lambda_{6}\lambda_{2} + \mu_{1}\mu_{5}\mu_{2} + \mu_{2}\lambda_{5}\mu_{2} + \mu_{3}\mu_{5}\lambda_{2} + \mu_{4}\lambda_{5}\lambda_{2} + \mu_{1}\mu_{7}\mu_{6} + \mu_{2}\lambda_{7}\mu_{6} + \mu_{3}\mu_{7}\lambda_{6} + \mu_{4}\lambda_{7}\lambda_{6} \\
f_{67} &= - \lambda_{5}\mu_{3} - \lambda_{6}\lambda_{3} + \mu_{1}\mu_{6}\mu_{1} + \mu_{2}\lambda_{6}\mu_{1} + \mu_{3}\mu_{6}\lambda_{1} + \mu_{4}\lambda_{6}\lambda_{1} + \mu_{1}\mu_{8}\mu_{5} + \mu_{2}\lambda_{8}\mu_{5} + \mu_{3}\mu_{8}\lambda_{5} + \mu_{4}\lambda_{8}\lambda_{5} \\
f_{68} &= - \lambda_{5}\mu_{4} - \lambda_{6}\lambda_{4} + \mu_{1}\mu_{5}\mu_{1} + \mu_{2}\lambda_{5}\mu_{1} + \mu_{3}\mu_{5}\lambda_{1} + \mu_{4}\lambda_{5}\lambda_{1} + \mu_{1}\mu_{7}\mu_{5} + \mu_{2}\lambda_{7}\mu_{5} + \mu_{3}\mu_{7}\lambda_{5} + \mu_{4}\lambda_{7}\lambda_{5} \\
f_{69} &= - \lambda_{5}\mu_{5} - \lambda_{6}\lambda_{5} + \mu_{5}\mu_{6}\mu_{2} + \mu_{6}\lambda_{6}\mu_{2} + \mu_{7}\mu_{6}\lambda_{2} + \mu_{8}\lambda_{6}\lambda_{2} + \mu_{5}\mu_{8}\mu_{6} + \mu_{6}\lambda_{8}\mu_{6} + \mu_{7}\mu_{8}\lambda_{6} + \mu_{8}\lambda_{8}\lambda_{6} \\
f_{70} &= - \lambda_{5}\mu_{6} - \lambda_{6}\lambda_{6} + \mu_{5}\mu_{5}\mu_{2} + \mu_{6}\lambda_{5}\mu_{2} + \mu_{7}\mu_{5}\lambda_{2} + \mu_{8}\lambda_{5}\lambda_{2} + \mu_{5}\mu_{7}\mu_{6} + \mu_{6}\lambda_{7}\mu_{6} + \mu_{7}\mu_{7}\lambda_{6} + \mu_{8}\lambda_{7}\lambda_{6} \\
f_{71} &= - \lambda_{5}\mu_{7} - \lambda_{6}\lambda_{7} + \mu_{5}\mu_{6}\mu_{1} + \mu_{6}\lambda_{6}\mu_{1} + \mu_{7}\mu_{6}\lambda_{1} + \mu_{8}\lambda_{6}\lambda_{1} + \mu_{5}\mu_{8}\mu_{5} + \mu_{6}\lambda_{8}\mu_{5} + \mu_{7}\mu_{8}\lambda_{5} + \mu_{8}\lambda_{8}\lambda_{5} \\
f_{72} &= - \lambda_{5}\mu_{8} - \lambda_{6}\lambda_{8} + \mu_{5}\mu_{5}\mu_{1} + \mu_{6}\lambda_{5}\mu_{1} + \mu_{7}\mu_{5}\lambda_{1} + \mu_{8}\lambda_{5}\lambda_{1} + \mu_{5}\mu_{7}\mu_{5} + \mu_{6}\lambda_{7}\mu_{5} + \mu_{7}\mu_{7}\lambda_{5} + \mu_{8}\lambda_{7}\lambda_{5} \\
f_{73} &= - \lambda_{7}\mu_{1} - \lambda_{8}\lambda_{1} + \mu_{1}\mu_{6}\mu_{4} + \mu_{2}\lambda_{6}\mu_{4} + \mu_{3}\mu_{6}\lambda_{4} + \mu_{4}\lambda_{6}\lambda_{4} + \mu_{1}\mu_{8}\mu_{8} + \mu_{2}\lambda_{8}\mu_{8} + \mu_{3}\mu_{8}\lambda_{8} + \mu_{4}\lambda_{8}\lambda_{8} \\
f_{74} &= - \lambda_{7}\mu_{2} - \lambda_{8}\lambda_{2} + \mu_{1}\mu_{5}\mu_{4} + \mu_{2}\lambda_{5}\mu_{4} + \mu_{3}\mu_{5}\lambda_{4} + \mu_{4}\lambda_{5}\lambda_{4} + \mu_{1}\mu_{7}\mu_{8} + \mu_{2}\lambda_{7}\mu_{8} + \mu_{3}\mu_{7}\lambda_{8} + \mu_{4}\lambda_{7}\lambda_{8} \\
f_{75} &= - \lambda_{7}\mu_{3} - \lambda_{8}\lambda_{3} + \mu_{1}\mu_{6}\mu_{3} + \mu_{2}\lambda_{6}\mu_{3} + \mu_{3}\mu_{6}\lambda_{3} + \mu_{4}\lambda_{6}\lambda_{3} + \mu_{1}\mu_{8}\mu_{7} + \mu_{2}\lambda_{8}\mu_{7} + \mu_{3}\mu_{8}\lambda_{7} + \mu_{4}\lambda_{8}\lambda_{7} \\
f_{76} &= - \lambda_{7}\mu_{4} - \lambda_{8}\lambda_{4} + \mu_{1}\mu_{5}\mu_{3} + \mu_{2}\lambda_{5}\mu_{3} + \mu_{3}\mu_{5}\lambda_{3} + \mu_{4}\lambda_{5}\lambda_{3} + \mu_{1}\mu_{7}\mu_{7} + \mu_{2}\lambda_{7}\mu_{7} + \mu_{3}\mu_{7}\lambda_{7} + \mu_{4}\lambda_{7}\lambda_{7} \\
f_{77} &= - \lambda_{7}\mu_{5} - \lambda_{8}\lambda_{5} + \mu_{5}\mu_{6}\mu_{4} + \mu_{6}\lambda_{6}\mu_{4} + \mu_{7}\mu_{6}\lambda_{4} + \mu_{8}\lambda_{6}\lambda_{4} + \mu_{5}\mu_{8}\mu_{8} + \mu_{6}\lambda_{8}\mu_{8} + \mu_{7}\mu_{8}\lambda_{8} + \mu_{8}\lambda_{8}\lambda_{8} \\
f_{78} &= - \lambda_{7}\mu_{6} - \lambda_{8}\lambda_{6} + \mu_{5}\mu_{5}\mu_{4} + \mu_{6}\lambda_{5}\mu_{4} + \mu_{7}\mu_{5}\lambda_{4} + \mu_{8}\lambda_{5}\lambda_{4} + \mu_{5}\mu_{7}\mu_{8} + \mu_{6}\lambda_{7}\mu_{8} + \mu_{7}\mu_{7}\lambda_{8} + \mu_{8}\lambda_{7}\lambda_{8} \\
f_{79} &= - \lambda_{7}\mu_{7} - \lambda_{8}\lambda_{7} + \mu_{5}\mu_{6}\mu_{3} + \mu_{6}\lambda_{6}\mu_{3} + \mu_{7}\mu_{6}\lambda_{3} + \mu_{8}\lambda_{6}\lambda_{3} + \mu_{5}\mu_{8}\mu_{7} + \mu_{6}\lambda_{8}\mu_{7} + \mu_{7}\mu_{8}\lambda_{7} + \mu_{8}\lambda_{8}\lambda_{7} \\
f_{80} &= - \lambda_{7}\mu_{8} - \lambda_{8}\lambda_{8} + \mu_{5}\mu_{5}\mu_{3} + \mu_{6}\lambda_{5}\mu_{3} + \mu_{7}\mu_{5}\lambda_{3} + \mu_{8}\lambda_{5}\lambda_{3} + \mu_{5}\mu_{7}\mu_{7} + \mu_{6}\lambda_{7}\mu_{7} + \mu_{7}\mu_{7}\lambda_{7} + \mu_{8}\lambda_{7}\lambda_{7} 
\end{align*}
\begin{align*}
f_{81} = \lambda_{1}\mu_{1}\mu_{5} &+ \lambda_{2}\lambda_{1}\mu_{5} + \lambda_{3}\mu_{1}\lambda_{5} + \lambda_{4}\lambda_{1}\lambda_{5} + \lambda_{1}\mu_{3}\mu_{1} + \lambda_{2}\lambda_{3}\mu_{1} + \lambda_{3}\mu_{3}\lambda_{1} + \lambda_{4}\lambda_{3}\lambda_{1} \\ &+ \lambda_{5}\mu_{5}\mu_{7} + \lambda_{6}\lambda_{5}\mu_{7} + \lambda_{7}\mu_{5}\lambda_{7} + \lambda_{8}\lambda_{5}\lambda_{7} + \lambda_{5}\mu_{7}\mu_{3} + \lambda_{6}\lambda_{7}\mu_{3} + \lambda_{7}\mu_{7}\lambda_{3} + \lambda_{8}\lambda_{7}\lambda_{3} \\
f_{82} = \lambda_{1}\mu_{1}\mu_{6} &+ \lambda_{2}\lambda_{1}\mu_{6} + \lambda_{3}\mu_{1}\lambda_{6} + \lambda_{4}\lambda_{1}\lambda_{6} + \lambda_{1}\mu_{3}\mu_{2} + \lambda_{2}\lambda_{3}\mu_{2} + \lambda_{3}\mu_{3}\lambda_{2} + \lambda_{4}\lambda_{3}\lambda_{2} \\ &+ \lambda_{5}\mu_{6}\mu_{7} + \lambda_{6}\lambda_{6}\mu_{7} + \lambda_{7}\mu_{6}\lambda_{7} + \lambda_{8}\lambda_{6}\lambda_{7} + \lambda_{5}\mu_{8}\mu_{3} + \lambda_{6}\lambda_{8}\mu_{3} + \lambda_{7}\mu_{8}\lambda_{3} + \lambda_{8}\lambda_{8}\lambda_{3} \\
f_{83} = \lambda_{1}\mu_{1}\mu_{7} &+ \lambda_{2}\lambda_{1}\mu_{7} + \lambda_{3}\mu_{1}\lambda_{7} + \lambda_{4}\lambda_{1}\lambda_{7} + \lambda_{1}\mu_{3}\mu_{3} + \lambda_{2}\lambda_{3}\mu_{3} + \lambda_{3}\mu_{3}\lambda_{3} + \lambda_{4}\lambda_{3}\lambda_{3} \\ &+ \lambda_{5}\mu_{5}\mu_{5} + \lambda_{6}\lambda_{5}\mu_{5} + \lambda_{7}\mu_{5}\lambda_{5} + \lambda_{8}\lambda_{5}\lambda_{5} + \lambda_{5}\mu_{7}\mu_{1} + \lambda_{6}\lambda_{7}\mu_{1} + \lambda_{7}\mu_{7}\lambda_{1} + \lambda_{8}\lambda_{7}\lambda_{1} \\
f_{84} = \lambda_{1}\mu_{1}\mu_{8} &+ \lambda_{2}\lambda_{1}\mu_{8} + \lambda_{3}\mu_{1}\lambda_{8} + \lambda_{4}\lambda_{1}\lambda_{8} + \lambda_{1}\mu_{3}\mu_{4} + \lambda_{2}\lambda_{3}\mu_{4} + \lambda_{3}\mu_{3}\lambda_{4} + \lambda_{4}\lambda_{3}\lambda_{4} \\ &+ \lambda_{5}\mu_{6}\mu_{5} + \lambda_{6}\lambda_{6}\mu_{5} + \lambda_{7}\mu_{6}\lambda_{5} + \lambda_{8}\lambda_{6}\lambda_{5} + \lambda_{5}\mu_{8}\mu_{1} + \lambda_{6}\lambda_{8}\mu_{1} + \lambda_{7}\mu_{8}\lambda_{1} + \lambda_{8}\lambda_{8}\lambda_{1} \\
f_{85} = \lambda_{1}\mu_{2}\mu_{5} &+ \lambda_{2}\lambda_{2}\mu_{5} + \lambda_{3}\mu_{2}\lambda_{5} + \lambda_{4}\lambda_{2}\lambda_{5} + \lambda_{1}\mu_{4}\mu_{1} + \lambda_{2}\lambda_{4}\mu_{1} + \lambda_{3}\mu_{4}\lambda_{1} + \lambda_{4}\lambda_{4}\lambda_{1} \\ &+ \lambda_{5}\mu_{5}\mu_{8} + \lambda_{6}\lambda_{5}\mu_{8} + \lambda_{7}\mu_{5}\lambda_{8} + \lambda_{8}\lambda_{5}\lambda_{8} + \lambda_{5}\mu_{7}\mu_{4} + \lambda_{6}\lambda_{7}\mu_{4} + \lambda_{7}\mu_{7}\lambda_{4} + \lambda_{8}\lambda_{7}\lambda_{4} \\
f_{86} = \lambda_{1}\mu_{2}\mu_{6} &+ \lambda_{2}\lambda_{2}\mu_{6} + \lambda_{3}\mu_{2}\lambda_{6} + \lambda_{4}\lambda_{2}\lambda_{6} + \lambda_{1}\mu_{4}\mu_{2} + \lambda_{2}\lambda_{4}\mu_{2} + \lambda_{3}\mu_{4}\lambda_{2} + \lambda_{4}\lambda_{4}\lambda_{2} \\ &+ \lambda_{5}\mu_{6}\mu_{8} + \lambda_{6}\lambda_{6}\mu_{8} + \lambda_{7}\mu_{6}\lambda_{8} + \lambda_{8}\lambda_{6}\lambda_{8} + \lambda_{5}\mu_{8}\mu_{4} + \lambda_{6}\lambda_{8}\mu_{4} + \lambda_{7}\mu_{8}\lambda_{4} + \lambda_{8}\lambda_{8}\lambda_{4} \\
f_{87} = \lambda_{1}\mu_{2}\mu_{7} &+ \lambda_{2}\lambda_{2}\mu_{7} + \lambda_{3}\mu_{2}\lambda_{7} + \lambda_{4}\lambda_{2}\lambda_{7} + \lambda_{1}\mu_{4}\mu_{3} + \lambda_{2}\lambda_{4}\mu_{3} + \lambda_{3}\mu_{4}\lambda_{3} + \lambda_{4}\lambda_{4}\lambda_{3} \\ &+ \lambda_{5}\mu_{5}\mu_{6} + \lambda_{6}\lambda_{5}\mu_{6} + \lambda_{7}\mu_{5}\lambda_{6} + \lambda_{8}\lambda_{5}\lambda_{6} + \lambda_{5}\mu_{7}\mu_{2} + \lambda_{6}\lambda_{7}\mu_{2} + \lambda_{7}\mu_{7}\lambda_{2} + \lambda_{8}\lambda_{7}\lambda_{2} \\
f_{88} = \lambda_{1}\mu_{2}\mu_{8} &+ \lambda_{2}\lambda_{2}\mu_{8} + \lambda_{3}\mu_{2}\lambda_{8} + \lambda_{4}\lambda_{2}\lambda_{8} + \lambda_{1}\mu_{4}\mu_{4} + \lambda_{2}\lambda_{4}\mu_{4} + \lambda_{3}\mu_{4}\lambda_{4} + \lambda_{4}\lambda_{4}\lambda_{4} \\ &+ \lambda_{5}\mu_{6}\mu_{6} + \lambda_{6}\lambda_{6}\mu_{6} + \lambda_{7}\mu_{6}\lambda_{6} + \lambda_{8}\lambda_{6}\lambda_{6} + \lambda_{5}\mu_{8}\mu_{2} + \lambda_{6}\lambda_{8}\mu_{2} + \lambda_{7}\mu_{8}\lambda_{2} + \lambda_{8}\lambda_{8}\lambda_{2} \\
f_{89} = \lambda_{1}\mu_{5}\mu_{5} &+ \lambda_{2}\lambda_{5}\mu_{5} + \lambda_{3}\mu_{5}\lambda_{5} + \lambda_{4}\lambda_{5}\lambda_{5} + \lambda_{1}\mu_{7}\mu_{1} + \lambda_{2}\lambda_{7}\mu_{1} + \lambda_{3}\mu_{7}\lambda_{1} + \lambda_{4}\lambda_{7}\lambda_{1} \\ &+ \lambda_{5}\mu_{1}\mu_{7} + \lambda_{6}\lambda_{1}\mu_{7} + \lambda_{7}\mu_{1}\lambda_{7} + \lambda_{8}\lambda_{1}\lambda_{7} + \lambda_{5}\mu_{3}\mu_{3} + \lambda_{6}\lambda_{3}\mu_{3} + \lambda_{7}\mu_{3}\lambda_{3} + \lambda_{8}\lambda_{3}\lambda_{3} \\
f_{90} = \lambda_{1}\mu_{5}\mu_{6} &+ \lambda_{2}\lambda_{5}\mu_{6} + \lambda_{3}\mu_{5}\lambda_{6} + \lambda_{4}\lambda_{5}\lambda_{6} + \lambda_{1}\mu_{7}\mu_{2} + \lambda_{2}\lambda_{7}\mu_{2} + \lambda_{3}\mu_{7}\lambda_{2} + \lambda_{4}\lambda_{7}\lambda_{2} \\ &+ \lambda_{5}\mu_{2}\mu_{7} + \lambda_{6}\lambda_{2}\mu_{7} + \lambda_{7}\mu_{2}\lambda_{7} + \lambda_{8}\lambda_{2}\lambda_{7} + \lambda_{5}\mu_{4}\mu_{3} + \lambda_{6}\lambda_{4}\mu_{3} + \lambda_{7}\mu_{4}\lambda_{3} + \lambda_{8}\lambda_{4}\lambda_{3} \\
f_{91} = \lambda_{1}\mu_{5}\mu_{7} &+ \lambda_{2}\lambda_{5}\mu_{7} + \lambda_{3}\mu_{5}\lambda_{7} + \lambda_{4}\lambda_{5}\lambda_{7} + \lambda_{1}\mu_{7}\mu_{3} + \lambda_{2}\lambda_{7}\mu_{3} + \lambda_{3}\mu_{7}\lambda_{3} + \lambda_{4}\lambda_{7}\lambda_{3} \\ &+ \lambda_{5}\mu_{1}\mu_{5} + \lambda_{6}\lambda_{1}\mu_{5} + \lambda_{7}\mu_{1}\lambda_{5} + \lambda_{8}\lambda_{1}\lambda_{5} + \lambda_{5}\mu_{3}\mu_{1} + \lambda_{6}\lambda_{3}\mu_{1} + \lambda_{7}\mu_{3}\lambda_{1} + \lambda_{8}\lambda_{3}\lambda_{1} \\
f_{92} = \lambda_{1}\mu_{5}\mu_{8} &+ \lambda_{2}\lambda_{5}\mu_{8} + \lambda_{3}\mu_{5}\lambda_{8} + \lambda_{4}\lambda_{5}\lambda_{8} + \lambda_{1}\mu_{7}\mu_{4} + \lambda_{2}\lambda_{7}\mu_{4} + \lambda_{3}\mu_{7}\lambda_{4} + \lambda_{4}\lambda_{7}\lambda_{4} \\ &+ \lambda_{5}\mu_{2}\mu_{5} + \lambda_{6}\lambda_{2}\mu_{5} + \lambda_{7}\mu_{2}\lambda_{5} + \lambda_{8}\lambda_{2}\lambda_{5} + \lambda_{5}\mu_{4}\mu_{1} + \lambda_{6}\lambda_{4}\mu_{1} + \lambda_{7}\mu_{4}\lambda_{1} + \lambda_{8}\lambda_{4}\lambda_{1} \\
f_{93} = \lambda_{1}\mu_{6}\mu_{5} &+ \lambda_{2}\lambda_{6}\mu_{5} + \lambda_{3}\mu_{6}\lambda_{5} + \lambda_{4}\lambda_{6}\lambda_{5} + \lambda_{1}\mu_{8}\mu_{1} + \lambda_{2}\lambda_{8}\mu_{1} + \lambda_{3}\mu_{8}\lambda_{1} + \lambda_{4}\lambda_{8}\lambda_{1} \\ &+ \lambda_{5}\mu_{1}\mu_{8} + \lambda_{6}\lambda_{1}\mu_{8} + \lambda_{7}\mu_{1}\lambda_{8} + \lambda_{8}\lambda_{1}\lambda_{8} + \lambda_{5}\mu_{3}\mu_{4} + \lambda_{6}\lambda_{3}\mu_{4} + \lambda_{7}\mu_{3}\lambda_{4} + \lambda_{8}\lambda_{3}\lambda_{4} \\
f_{94} = \lambda_{1}\mu_{6}\mu_{6} &+ \lambda_{2}\lambda_{6}\mu_{6} + \lambda_{3}\mu_{6}\lambda_{6} + \lambda_{4}\lambda_{6}\lambda_{6} + \lambda_{1}\mu_{8}\mu_{2} + \lambda_{2}\lambda_{8}\mu_{2} + \lambda_{3}\mu_{8}\lambda_{2} + \lambda_{4}\lambda_{8}\lambda_{2} \\ &+ \lambda_{5}\mu_{2}\mu_{8} + \lambda_{6}\lambda_{2}\mu_{8} + \lambda_{7}\mu_{2}\lambda_{8} + \lambda_{8}\lambda_{2}\lambda_{8} + \lambda_{5}\mu_{4}\mu_{4} + \lambda_{6}\lambda_{4}\mu_{4} + \lambda_{7}\mu_{4}\lambda_{4} + \lambda_{8}\lambda_{4}\lambda_{4} \\
f_{95} = \lambda_{1}\mu_{6}\mu_{7} &+ \lambda_{2}\lambda_{6}\mu_{7} + \lambda_{3}\mu_{6}\lambda_{7} + \lambda_{4}\lambda_{6}\lambda_{7} + \lambda_{1}\mu_{8}\mu_{3} + \lambda_{2}\lambda_{8}\mu_{3} + \lambda_{3}\mu_{8}\lambda_{3} + \lambda_{4}\lambda_{8}\lambda_{3} \\ &+ \lambda_{5}\mu_{1}\mu_{6} + \lambda_{6}\lambda_{1}\mu_{6} + \lambda_{7}\mu_{1}\lambda_{6} + \lambda_{8}\lambda_{1}\lambda_{6} + \lambda_{5}\mu_{3}\mu_{2} + \lambda_{6}\lambda_{3}\mu_{2} + \lambda_{7}\mu_{3}\lambda_{2} + \lambda_{8}\lambda_{3}\lambda_{2} \\
f_{96} = \lambda_{1}\mu_{6}\mu_{8} &+ \lambda_{2}\lambda_{6}\mu_{8} + \lambda_{3}\mu_{6}\lambda_{8} + \lambda_{4}\lambda_{6}\lambda_{8} + \lambda_{1}\mu_{8}\mu_{4} + \lambda_{2}\lambda_{8}\mu_{4} + \lambda_{3}\mu_{8}\lambda_{4} + \lambda_{4}\lambda_{8}\lambda_{4} \\ &+ \lambda_{5}\mu_{2}\mu_{6} + \lambda_{6}\lambda_{2}\mu_{6} + \lambda_{7}\mu_{2}\lambda_{6} + \lambda_{8}\lambda_{2}\lambda_{6} + \lambda_{5}\mu_{4}\mu_{2} + \lambda_{6}\lambda_{4}\mu_{2} + \lambda_{7}\mu_{4}\lambda_{2} + \lambda_{8}\lambda_{4}\lambda_{2} 
\end{align*}
\begin{align*}
f_{97} &= - \mu_{1}\mu_{1} - \mu_{2}\lambda_{1} + \lambda_{1}\mu_{2}\mu_{2} + \lambda_{2}\lambda_{2}\mu_{2} + \lambda_{3}\mu_{2}\lambda_{2} + \lambda_{4}\lambda_{2}\lambda_{2} + \lambda_{1}\mu_{4}\mu_{6} + \lambda_{2}\lambda_{4}\mu_{6} + \lambda_{3}\mu_{4}\lambda_{6} + \lambda_{4}\lambda_{4}\lambda_{6} \\
f_{98} &= - \mu_{1}\mu_{2} - \mu_{2}\lambda_{2} + \lambda_{1}\mu_{1}\mu_{2} + \lambda_{2}\lambda_{1}\mu_{2} + \lambda_{3}\mu_{1}\lambda_{2} + \lambda_{4}\lambda_{1}\lambda_{2} + \lambda_{1}\mu_{3}\mu_{6} + \lambda_{2}\lambda_{3}\mu_{6} + \lambda_{3}\mu_{3}\lambda_{6} + \lambda_{4}\lambda_{3}\lambda_{6} \\
f_{99} &= - \mu_{1}\mu_{3} - \mu_{2}\lambda_{3} + \lambda_{1}\mu_{2}\mu_{1} + \lambda_{2}\lambda_{2}\mu_{1} + \lambda_{3}\mu_{2}\lambda_{1} + \lambda_{4}\lambda_{2}\lambda_{1} + \lambda_{1}\mu_{4}\mu_{5} + \lambda_{2}\lambda_{4}\mu_{5} + \lambda_{3}\mu_{4}\lambda_{5} + \lambda_{4}\lambda_{4}\lambda_{5} \\
f_{100} &= - \mu_{1}\mu_{4} - \mu_{2}\lambda_{4} + \lambda_{1}\mu_{1}\mu_{1} + \lambda_{2}\lambda_{1}\mu_{1} + \lambda_{3}\mu_{1}\lambda_{1} + \lambda_{4}\lambda_{1}\lambda_{1} + \lambda_{1}\mu_{3}\mu_{5} + \lambda_{2}\lambda_{3}\mu_{5} + \lambda_{3}\mu_{3}\lambda_{5} + \lambda_{4}\lambda_{3}\lambda_{5} \\
f_{101} &= - \mu_{1}\mu_{5} - \mu_{2}\lambda_{5} + \lambda_{5}\mu_{2}\mu_{2} + \lambda_{6}\lambda_{2}\mu_{2} + \lambda_{7}\mu_{2}\lambda_{2} + \lambda_{8}\lambda_{2}\lambda_{2} + \lambda_{5}\mu_{4}\mu_{6} + \lambda_{6}\lambda_{4}\mu_{6} + \lambda_{7}\mu_{4}\lambda_{6} + \lambda_{8}\lambda_{4}\lambda_{6} \\
f_{102} &= - \mu_{1}\mu_{6} - \mu_{2}\lambda_{6} + \lambda_{5}\mu_{1}\mu_{2} + \lambda_{6}\lambda_{1}\mu_{2} + \lambda_{7}\mu_{1}\lambda_{2} + \lambda_{8}\lambda_{1}\lambda_{2} + \lambda_{5}\mu_{3}\mu_{6} + \lambda_{6}\lambda_{3}\mu_{6} + \lambda_{7}\mu_{3}\lambda_{6} + \lambda_{8}\lambda_{3}\lambda_{6} \\
f_{103} &= - \mu_{1}\mu_{7} - \mu_{2}\lambda_{7} + \lambda_{5}\mu_{2}\mu_{1} + \lambda_{6}\lambda_{2}\mu_{1} + \lambda_{7}\mu_{2}\lambda_{1} + \lambda_{8}\lambda_{2}\lambda_{1} + \lambda_{5}\mu_{4}\mu_{5} + \lambda_{6}\lambda_{4}\mu_{5} + \lambda_{7}\mu_{4}\lambda_{5} + \lambda_{8}\lambda_{4}\lambda_{5} \\
f_{104} &= - \mu_{1}\mu_{8} - \mu_{2}\lambda_{8} + \lambda_{5}\mu_{1}\mu_{1} + \lambda_{6}\lambda_{1}\mu_{1} + \lambda_{7}\mu_{1}\lambda_{1} + \lambda_{8}\lambda_{1}\lambda_{1} + \lambda_{5}\mu_{3}\mu_{5} + \lambda_{6}\lambda_{3}\mu_{5} + \lambda_{7}\mu_{3}\lambda_{5} + \lambda_{8}\lambda_{3}\lambda_{5} \\
f_{105} &= - \mu_{3}\mu_{1} - \mu_{4}\lambda_{1} + \lambda_{1}\mu_{2}\mu_{4} + \lambda_{2}\lambda_{2}\mu_{4} + \lambda_{3}\mu_{2}\lambda_{4} + \lambda_{4}\lambda_{2}\lambda_{4} + \lambda_{1}\mu_{4}\mu_{8} + \lambda_{2}\lambda_{4}\mu_{8} + \lambda_{3}\mu_{4}\lambda_{8} + \lambda_{4}\lambda_{4}\lambda_{8} \\
f_{106} &= - \mu_{3}\mu_{2} - \mu_{4}\lambda_{2} + \lambda_{1}\mu_{1}\mu_{4} + \lambda_{2}\lambda_{1}\mu_{4} + \lambda_{3}\mu_{1}\lambda_{4} + \lambda_{4}\lambda_{1}\lambda_{4} + \lambda_{1}\mu_{3}\mu_{8} + \lambda_{2}\lambda_{3}\mu_{8} + \lambda_{3}\mu_{3}\lambda_{8} + \lambda_{4}\lambda_{3}\lambda_{8} \\
f_{107} &= - \mu_{3}\mu_{3} - \mu_{4}\lambda_{3} + \lambda_{1}\mu_{2}\mu_{3} + \lambda_{2}\lambda_{2}\mu_{3} + \lambda_{3}\mu_{2}\lambda_{3} + \lambda_{4}\lambda_{2}\lambda_{3} + \lambda_{1}\mu_{4}\mu_{7} + \lambda_{2}\lambda_{4}\mu_{7} + \lambda_{3}\mu_{4}\lambda_{7} + \lambda_{4}\lambda_{4}\lambda_{7} \\
f_{108} &= - \mu_{3}\mu_{4} - \mu_{4}\lambda_{4} + \lambda_{1}\mu_{1}\mu_{3} + \lambda_{2}\lambda_{1}\mu_{3} + \lambda_{3}\mu_{1}\lambda_{3} + \lambda_{4}\lambda_{1}\lambda_{3} + \lambda_{1}\mu_{3}\mu_{7} + \lambda_{2}\lambda_{3}\mu_{7} + \lambda_{3}\mu_{3}\lambda_{7} + \lambda_{4}\lambda_{3}\lambda_{7} \\
f_{109} &= - \mu_{3}\mu_{5} - \mu_{4}\lambda_{5} + \lambda_{5}\mu_{2}\mu_{4} + \lambda_{6}\lambda_{2}\mu_{4} + \lambda_{7}\mu_{2}\lambda_{4} + \lambda_{8}\lambda_{2}\lambda_{4} + \lambda_{5}\mu_{4}\mu_{8} + \lambda_{6}\lambda_{4}\mu_{8} + \lambda_{7}\mu_{4}\lambda_{8} + \lambda_{8}\lambda_{4}\lambda_{8} \\
f_{110} &= - \mu_{3}\mu_{6} - \mu_{4}\lambda_{6} + \lambda_{5}\mu_{1}\mu_{4} + \lambda_{6}\lambda_{1}\mu_{4} + \lambda_{7}\mu_{1}\lambda_{4} + \lambda_{8}\lambda_{1}\lambda_{4} + \lambda_{5}\mu_{3}\mu_{8} + \lambda_{6}\lambda_{3}\mu_{8} + \lambda_{7}\mu_{3}\lambda_{8} + \lambda_{8}\lambda_{3}\lambda_{8} \\
f_{111} &= - \mu_{3}\mu_{7} - \mu_{4}\lambda_{7} + \lambda_{5}\mu_{2}\mu_{3} + \lambda_{6}\lambda_{2}\mu_{3} + \lambda_{7}\mu_{2}\lambda_{3} + \lambda_{8}\lambda_{2}\lambda_{3} + \lambda_{5}\mu_{4}\mu_{7} + \lambda_{6}\lambda_{4}\mu_{7} + \lambda_{7}\mu_{4}\lambda_{7} + \lambda_{8}\lambda_{4}\lambda_{7} \\
f_{112} &= - \mu_{3}\mu_{8} - \mu_{4}\lambda_{8} + \lambda_{5}\mu_{1}\mu_{3} + \lambda_{6}\lambda_{1}\mu_{3} + \lambda_{7}\mu_{1}\lambda_{3} + \lambda_{8}\lambda_{1}\lambda_{3} + \lambda_{5}\mu_{3}\mu_{7} + \lambda_{6}\lambda_{3}\mu_{7} + \lambda_{7}\mu_{3}\lambda_{7} + \lambda_{8}\lambda_{3}\lambda_{7} \\
f_{113} &= - \mu_{5}\mu_{1} - \mu_{6}\lambda_{1} + \lambda_{1}\mu_{6}\mu_{2} + \lambda_{2}\lambda_{6}\mu_{2} + \lambda_{3}\mu_{6}\lambda_{2} + \lambda_{4}\lambda_{6}\lambda_{2} + \lambda_{1}\mu_{8}\mu_{6} + \lambda_{2}\lambda_{8}\mu_{6} + \lambda_{3}\mu_{8}\lambda_{6} + \lambda_{4}\lambda_{8}\lambda_{6} \\
f_{114} &= - \mu_{5}\mu_{2} - \mu_{6}\lambda_{2} + \lambda_{1}\mu_{5}\mu_{2} + \lambda_{2}\lambda_{5}\mu_{2} + \lambda_{3}\mu_{5}\lambda_{2} + \lambda_{4}\lambda_{5}\lambda_{2} + \lambda_{1}\mu_{7}\mu_{6} + \lambda_{2}\lambda_{7}\mu_{6} + \lambda_{3}\mu_{7}\lambda_{6} + \lambda_{4}\lambda_{7}\lambda_{6} \\
f_{115} &= - \mu_{5}\mu_{3} - \mu_{6}\lambda_{3} + \lambda_{1}\mu_{6}\mu_{1} + \lambda_{2}\lambda_{6}\mu_{1} + \lambda_{3}\mu_{6}\lambda_{1} + \lambda_{4}\lambda_{6}\lambda_{1} + \lambda_{1}\mu_{8}\mu_{5} + \lambda_{2}\lambda_{8}\mu_{5} + \lambda_{3}\mu_{8}\lambda_{5} + \lambda_{4}\lambda_{8}\lambda_{5} \\
f_{116} &= - \mu_{5}\mu_{4} - \mu_{6}\lambda_{4} + \lambda_{1}\mu_{5}\mu_{1} + \lambda_{2}\lambda_{5}\mu_{1} + \lambda_{3}\mu_{5}\lambda_{1} + \lambda_{4}\lambda_{5}\lambda_{1} + \lambda_{1}\mu_{7}\mu_{5} + \lambda_{2}\lambda_{7}\mu_{5} + \lambda_{3}\mu_{7}\lambda_{5} + \lambda_{4}\lambda_{7}\lambda_{5} \\
f_{117} &= - \mu_{5}\mu_{5} - \mu_{6}\lambda_{5} + \lambda_{5}\mu_{6}\mu_{2} + \lambda_{6}\lambda_{6}\mu_{2} + \lambda_{7}\mu_{6}\lambda_{2} + \lambda_{8}\lambda_{6}\lambda_{2} + \lambda_{5}\mu_{8}\mu_{6} + \lambda_{6}\lambda_{8}\mu_{6} + \lambda_{7}\mu_{8}\lambda_{6} + \lambda_{8}\lambda_{8}\lambda_{6} \\
f_{118} &= - \mu_{5}\mu_{6} - \mu_{6}\lambda_{6} + \lambda_{5}\mu_{5}\mu_{2} + \lambda_{6}\lambda_{5}\mu_{2} + \lambda_{7}\mu_{5}\lambda_{2} + \lambda_{8}\lambda_{5}\lambda_{2} + \lambda_{5}\mu_{7}\mu_{6} + \lambda_{6}\lambda_{7}\mu_{6} + \lambda_{7}\mu_{7}\lambda_{6} + \lambda_{8}\lambda_{7}\lambda_{6} \\
f_{119} &= - \mu_{5}\mu_{7} - \mu_{6}\lambda_{7} + \lambda_{5}\mu_{6}\mu_{1} + \lambda_{6}\lambda_{6}\mu_{1} + \lambda_{7}\mu_{6}\lambda_{1} + \lambda_{8}\lambda_{6}\lambda_{1} + \lambda_{5}\mu_{8}\mu_{5} + \lambda_{6}\lambda_{8}\mu_{5} + \lambda_{7}\mu_{8}\lambda_{5} + \lambda_{8}\lambda_{8}\lambda_{5} \\
f_{120} &= - \mu_{5}\mu_{8} - \mu_{6}\lambda_{8} + \lambda_{5}\mu_{5}\mu_{1} + \lambda_{6}\lambda_{5}\mu_{1} + \lambda_{7}\mu_{5}\lambda_{1} + \lambda_{8}\lambda_{5}\lambda_{1} + \lambda_{5}\mu_{7}\mu_{5} + \lambda_{6}\lambda_{7}\mu_{5} + \lambda_{7}\mu_{7}\lambda_{5} + \lambda_{8}\lambda_{7}\lambda_{5} \\
f_{121} &= - \mu_{7}\mu_{1} - \mu_{8}\lambda_{1} + \lambda_{1}\mu_{6}\mu_{4} + \lambda_{2}\lambda_{6}\mu_{4} + \lambda_{3}\mu_{6}\lambda_{4} + \lambda_{4}\lambda_{6}\lambda_{4} + \lambda_{1}\mu_{8}\mu_{8} + \lambda_{2}\lambda_{8}\mu_{8} + \lambda_{3}\mu_{8}\lambda_{8} + \lambda_{4}\lambda_{8}\lambda_{8} \\
f_{122} &= - \mu_{7}\mu_{2} - \mu_{8}\lambda_{2} + \lambda_{1}\mu_{5}\mu_{4} + \lambda_{2}\lambda_{5}\mu_{4} + \lambda_{3}\mu_{5}\lambda_{4} + \lambda_{4}\lambda_{5}\lambda_{4} + \lambda_{1}\mu_{7}\mu_{8} + \lambda_{2}\lambda_{7}\mu_{8} + \lambda_{3}\mu_{7}\lambda_{8} + \lambda_{4}\lambda_{7}\lambda_{8} \\
f_{123} &= - \mu_{7}\mu_{3} - \mu_{8}\lambda_{3} + \lambda_{1}\mu_{6}\mu_{3} + \lambda_{2}\lambda_{6}\mu_{3} + \lambda_{3}\mu_{6}\lambda_{3} + \lambda_{4}\lambda_{6}\lambda_{3} + \lambda_{1}\mu_{8}\mu_{7} + \lambda_{2}\lambda_{8}\mu_{7} + \lambda_{3}\mu_{8}\lambda_{7} + \lambda_{4}\lambda_{8}\lambda_{7} \\
f_{124} &= - \mu_{7}\mu_{4} - \mu_{8}\lambda_{4} + \lambda_{1}\mu_{5}\mu_{3} + \lambda_{2}\lambda_{5}\mu_{3} + \lambda_{3}\mu_{5}\lambda_{3} + \lambda_{4}\lambda_{5}\lambda_{3} + \lambda_{1}\mu_{7}\mu_{7} + \lambda_{2}\lambda_{7}\mu_{7} + \lambda_{3}\mu_{7}\lambda_{7} + \lambda_{4}\lambda_{7}\lambda_{7} \\
f_{125} &= - \mu_{7}\mu_{5} - \mu_{8}\lambda_{5} + \lambda_{5}\mu_{6}\mu_{4} + \lambda_{6}\lambda_{6}\mu_{4} + \lambda_{7}\mu_{6}\lambda_{4} + \lambda_{8}\lambda_{6}\lambda_{4} + \lambda_{5}\mu_{8}\mu_{8} + \lambda_{6}\lambda_{8}\mu_{8} + \lambda_{7}\mu_{8}\lambda_{8} + \lambda_{8}\lambda_{8}\lambda_{8} \\
f_{126} &= - \mu_{7}\mu_{6} - \mu_{8}\lambda_{6} + \lambda_{5}\mu_{5}\mu_{4} + \lambda_{6}\lambda_{5}\mu_{4} + \lambda_{7}\mu_{5}\lambda_{4} + \lambda_{8}\lambda_{5}\lambda_{4} + \lambda_{5}\mu_{7}\mu_{8} + \lambda_{6}\lambda_{7}\mu_{8} + \lambda_{7}\mu_{7}\lambda_{8} + \lambda_{8}\lambda_{7}\lambda_{8} \\
f_{127} &= - \mu_{7}\mu_{7} - \mu_{8}\lambda_{7} + \lambda_{5}\mu_{6}\mu_{3} + \lambda_{6}\lambda_{6}\mu_{3} + \lambda_{7}\mu_{6}\lambda_{3} + \lambda_{8}\lambda_{6}\lambda_{3} + \lambda_{5}\mu_{8}\mu_{7} + \lambda_{6}\lambda_{8}\mu_{7} + \lambda_{7}\mu_{8}\lambda_{7} + \lambda_{8}\lambda_{8}\lambda_{7} \\
f_{128} &= - \mu_{7}\mu_{8} - \mu_{8}\lambda_{8} + \lambda_{5}\mu_{5}\mu_{3} + \lambda_{6}\lambda_{5}\mu_{3} + \lambda_{7}\mu_{5}\lambda_{3} + \lambda_{8}\lambda_{5}\lambda_{3} + \lambda_{5}\mu_{7}\mu_{7} + \lambda_{6}\lambda_{7}\mu_{7} + \lambda_{7}\mu_{7}\lambda_{7} + \lambda_{8}\lambda_{7}\lambda_{7}
\end{align*}

\normalsize

\end{document}